%% file: main.tex
\documentclass[a4paper]{article}         % <---- use this line to produce final pdf
\input{macros_draft.tex}
\input{macros.tex}

\title{
  \vspace{-1.5cm}
  The stable converse soul question for positively curved homogeneous spaces}
\author{David González-Álvaro \& Marcus Zibrowius}
\footnotedate{\today{}}
\date{}
\subjclass{
53C21  % Differential geometry
(19L64,  % K-theory
57R22)  % Manifolds and cell complexes
}
\grant{D.\ González-Álvaro received support from research grants MTM2014-57769-3-P and MTM2014-57309-REDT (MINECO).}

\begin{document}
\maketitle
\thispagestyle{empty}

\begin{abstract}
\noindent
The stable converse soul question (SCSQ) asks whether, given a real vector bundle \(E\) over a compact manifold, some stabilization \(E\times\R^k\) admits a metric with non-negative (sectional) curvature. We extend previous results to show that the SCSQ has an affirmative answer for all real vector bundles over any simply connected homogeneous manifold with positive curvature, except possibly for the Berger space \(B^{13}\).
Along the way, we show that the same is true for all simply connected homogeneous spaces of dimension at most seven, for arbitrary products of simply connected compact rank one symmetric spaces of dimensions multiples of four, and for certain products of spheres.  Moreover, we observe that the SCSQ is ``stable under tangential homotopy equivalence'':  if it has an affirmative answer for all vector bundles over a certain manifold \(M\), then the same is true for any manifold tangentially homotopy equivalent to~\(M\).  Our main tool is topological K-theory.  Over \(B^{13}\), there is essentially one stable class of real vector bundles for which our method fails.

\renewcommand{\contentsname}{\vspace{-0.5cm}}
\tableofcontents
\end{abstract}
\bigskip

\section{Introduction and statement of results}
\input{Intro.tex}\label{S:Intro}

\section{K-theory}\label{S:K-theory}
\input{KTheory.tex}

\section{Homogeneous spaces and bundles}
\input{Homogeneous.tex}

\section{Curvature on vector bundles}
\input{Curvature.tex}

\section{The positive cases}
\input{PositiveCases.tex}

\section{The 13-dimensional Berger manifold}\label{S:B13}
\input{Berger13.tex}

\end{document}

%% file: macros_draft.tex
\usepackage{ifdraft}   % breaks hyperlinks package if included here, see macros.tex instead
\usepackage{marginfix} % does not work here, see macros.tex instead

\ifdraft{%
  \usepackage[left=4cm,right=4cm,top=1cm,bottom=3cm,marginparwidth=3cm]{geometry}
}{}

\newcounter{commentcounter}[section]

%% file: macros.tex
\usepackage[utf8]{inputenc}
\usepackage[T1]{fontenc}
\usepackage[british]{babel}
\usepackage{
  etex,  % auxiliary package;  necessary when working with many packages
  lmodern,
  microtype,
  calc,
  amssymb,
  mathtools,
  stmaryrd,
  amsthm,
  units,
  paralist,
  tabularx,
  booktabs,
  multicol,
  graphicx,
  booktabs,
  changepage,%adjust margins (for tables)
}
\mathtoolsset{mathic}
\usepackage[dvipsnames]{xcolor}
\usepackage[draft=false]{hyperref}
\hypersetup{breaklinks=true,colorlinks=true,linkcolor=black,anchorcolor=black,citecolor=black,filecolor=black,menucolor=black,urlcolor=MidnightBlue}
\usepackage[hyperpageref]{backref}
\usepackage[alphabetic,non-sorted-cites]{amsrefs}
\usepackage[nameinlink,capitalize]{cleveref}
\makeatletter
\DeclareRobustCommand{\crefnosort}[1]{%
  \begingroup\@cref@sortfalse\cref{#1}\endgroup
}
\makeatother
\usepackage[all]{xy}
\makeatletter
\newcommand{\subjclass}[2][2010]{%
  \let\@oldtitle\@title%
  \gdef\@title{\@oldtitle\footnotetext{\hspace{0.3cm}\emph{#1 Mathematics Subject Classification:} #2}}%
}
\newcommand{\grant}[1]{%
  \let\@@oldtitle\@title%
  \gdef\@title{\@@oldtitle\footnotetext{\hspace{0.3cm}\parbox{0.85\linewidth}{#1}}}%
}
\newcommand{\keywords}[1]{%
  \let\@@@oldtitle\@title%
  \gdef\@title{\@@@oldtitle\footnotetext{\hspace{0.3cm}\emph{Key words and phrases:} #1}}%
}
\newcommand{\footnotedate}[1]{%
  \let\@@@@oldtitle\@title%
  \gdef\@title{\@@@@oldtitle\footnotetext{\hspace{0.3cm}\emph{Date:} #1}}%
}
\makeatother
\usepackage{tikz,tikz-cd,etex}
\usetikzlibrary{scopes,intersections,matrix,arrows,decorations.markings,calc}%,babel}
\tikzset{
  >=cm to,
  descr/.style={fill=white,inner sep=1.5pt},
  grid/.style={matrix of math nodes, row sep=2em, column sep=2.5em, text height=1.5ex, text depth=0.25ex},
  arrow/.style={overlay,->, font=\scriptsize},
}

\chardef\bslash=`\\ % p. 424, TeXbook
\mathchardef\mhyphen="2D

\newtheorem{thm}{Theorem}[section]\crefname{thm}{Theorem}{Theorems}
\newtheorem*{thm*}{Theorem}
\crefname{conj}{Conjecture}{Conjectures}
\newtheorem*{conj*}{Conjecture}
\newtheorem{cor}[thm]{Corollary}\crefname{cor}{Corollary}{Corollaries}
\newtheorem*{cor*}{Corollary}
\newtheorem{lem}[thm]{Lemma}\crefname{lem}{Lemma}{Lemmas}
\newtheorem*{lem*}{Lemma}
\newtheorem{prop}[thm]{Proposition}\crefname{prop}{Proposition}{Propositions}
\newtheorem*{prop*}{Proposition}
\crefname{idea}{Idea}{Ideas}
\newtheorem*{idea*}{Idea}
\crefname{defn}{Definition}{Definitions}
\newtheorem*{defn*}{Definition}

\newtheorem*{question*}{Question}

\newtheorem*{mainthm}{Main Theorem}%\crefname{mainthm}{the Main Theorem}{}
\newtheorem*{mainthmcont}{Main Theorem (cont.)}%\crefname{mainthm}{the Main Theorem}{}
\newtheorem{workhorse}[thm]{``Work Horse Theorem''}\crefname{workhorse}{the ``Work Horse Theorem''}{}\Crefname{workhorse}{The ``Work Horse Theorem''}{}

\theoremstyle{remark}
\newtheorem{rem}[thm]{Remark}\crefname{rem}{Remark}{Remarks}
\newtheorem*{rem*}{Remark}
\newtheorem{example}[thm]{Example}\crefname{eg}{Example}{Examples}
\newtheorem*{example*}{Example}
\newtheorem{examples}[thm]{Examples}\crefname{egs}{Examples}{Examples}
\newtheorem*{examples*}{Examples}

\newcommand*{\closedhookrightarrow}{\mathrel{\ooalign{$\hookrightarrow$\cr\hidewidth\raise0.1ex\hbox{$\rule{0.5pt}{1ex}\;\;$}\cr}}}

\newcommand*{\closedhookleftarrow}{\mathrel{\ooalign{$\hookleftarrow$\cr\hidewidth\raise0.1ex\hbox{$\rule{0.5pt}{1ex}\,\;$}\cr}}}

\makeatletter
\newcommand\fg{f.\,g.\@ifnextchar){}{\ }} % f.g. = finitely generated
\newcommand\fd{f.\,d.\@ifnextchar){}{\ }} % f.d. = finite dimensional
\makeatother
\newcommand{\ie}{i.\,e.\ }
\newcommand{\eg}{e.\,g.\ }
\newcommand{\cf}{c.\,f.\ }

\newcommand*{\F}{\mathbb{F}}
\newcommand*{\C}{\mathbb{C}}
\newcommand*{\R}{\mathbb{R}}
\newcommand*{\Z}{\mathbb{Z}}
\newcommand*{\N}{\mathbb{N}}

\newcommand{\cx}{\mathsf{c}}
\newcommand{\rx}{\mathsf{r}}
\newcommand{\tx}{\mathsf{t}}
\newcommand{\qx}{\mathsf{q}}

\renewcommand*{\SS}{\mathbb{S}} % the "round" spheres
\newcommand*{\CC}{\mathbb{C}}
\newcommand*{\NN}{\mathbb{N}}

\newcommand*{\QQ}{\mathbb{Q}}
\newcommand*{\RR}{\mathbb{R}}
\newcommand*{\ZZ}{\mathbb{Z}}
\newcommand*{\HH}{\mathbb{H}}

\newcommand{\CP}{\mathbb{CP}}
\newcommand{\RP}{\mathbb{RP}}
\newcommand{\HP}{\mathbb{HP}}

\newcommand*{\paircolumn}[2]{\binom{#1}{#2}}

\newcommand{\reduced}[1]{\smash{\resizebox{!}{2ex}{$\widetilde{#1}$}}}

\newcommand{\id}{\mathrm{id}}
\newcommand{\Spin}{Spin}

\DeclareMathOperator{\Rep}{Rep}
\DeclareMathOperator{\Vect}{Vect}
\DeclareMathOperator{\diag}{diag}

\DeclareMathOperator{\Ext}{Ext}
\DeclareMathOperator{\Hom}{Hom}

\DeclareMathOperator{\im}{im}
\DeclareMathOperator{\rank}{rank}

\DeclareMathOperator{\projdim}{\mathrm{pd}}

%% file: Intro.tex
The \textit{Soul Theorem} by Cheeger and Gromoll \cite{CG:on} determines the structure of complete open manifolds with non-negative sectional curvature:
given such a manifold \(M\), there exists a totally geodesic, totally convex compact submanifold \(S\) (called the soul) such that \(M\) is diffeomorphic to the normal bundle of \(S\).
The \textit{Converse Soul Question} asks, conversely, which vector bundles over a compact non-negatively curved manifold \(S\) admit a non-negatively curved metric.
In general, this question is still widely open.
The only such vector bundles that are known \emph{not} to admit a non-negatively curved metric occur over manifolds \(S\) with infinite fundamental group \cites{BK:obstructions,OW:vector}.
On the other hand, the only simply connected manifolds over which \emph{all} vector bundles are known to admit a non-negatively curved metric are the spheres \(\SS^n\) with \(n\leq 5\) \cite{GroveZiller:Milnor}.
Other than that, there are only partial results, and a weaker question has been studied, which we
formulate as follows.
From now on, \(M\) will denote a (smooth) compact manifold.

\begin{question*}[Stable Converse Soul Question, SCSQ]
Let \(E\) be the total space of a real vector bundle over a compact manifold \(M\).
Is there an integer \(k\geq 0\) such that the manifold \(E\times\R^k\) admits a metric of non-negative sectional curvature?
\end{question*}

For simply connected \(M\) with non-negative sectional curvature, this is precisely Problem~5.4 of \cite{AIMpl}. In our formulation, we do not explicitly require any curvature related conditions on~\(M\).
Note however that, as a consequence of the Soul Theorem, one can expect positive results for the SCSQ only when \(M\) is at least homotopy equivalent to a compact manifold with non-negative curvature.
The existence of manifolds over which the SCSQ has a negative answer for \emph{all} vector bundles is therefore immediate: simply consider compact manifolds whose Betti numbers violate Gromov's bound for the existence of metrics with non-negative curvature  \cite{Gromov:curvature}, for example connected sums \(\#^m\CP^n\) for sufficiently large \(m\) and arbitrary \(n\).

On the positive side, the SCSQ is already known to have an affirmative answer for all real vector bundles over any sphere \(\SS^n\) \cite{R:geodesics}, and for all real vector bundles over \(\CP^2\), \(\SS^2\times\SS^2\) and \(\CP^2\#(-\CP^2)\) \cite{GroveZiller:Lifting}. In \cite{Gonzalez:nonnegative},  the first-named author extended these results to include all compact rank one symmetric spaces (CROSSes): the spheres \(\SS^n\), the projective spaces \(\RP^n\), \(\CP^n\) and \(\HP^n\), and the Cayley plane.

In this article, we study the SCSQ over the larger class of compact simply connected homogeneous manifolds with an invariant metric of positive sectional curvature.
These manifolds are completely classified (see the recent classification \cite{WZ:Revisitng}).
In addition to the CROSSes they include the Wallach flag manifolds \(W^6\), \(W^{12}\) and \(W^{24}\), the Berger spaces \(B^7\) and \(B^{13}\), and the infinite family of Aloff-Wallach spaces \(W^{7}_{p,q}\).
Our main result is the following.

\begin{mainthm}%\label{THM: SCST for positive curvature 1}
The SCSQ has a positive answer for all real vector bundles over any compact simply connected homogeneous manifold \(M\) with an invariant metric of positive sectional curvature, except possibly for \(M \cong B^{13}\).
\end{mainthm}

In fact, we can show a more general statement.
Recall that two manifolds \(M,N\) of the same dimension are defined to be tangentially homotopy equivalent if there exists a homotopy equivalence \(f\colon M\to N\) such that the tangent bundle \(TM\) and \(f^*TN\) are stably isomorphic, \ie such that \(TM\times\R^k\) and \(f^*TN\times\R^k\) are isomorphic as bundles over \(M\) for some integer \(k\geq 0\).

\begin{mainthmcont}
More generally, the SCSQ has a positive answer for any real vector bundle over any compact manifold tangentially homotopy equivalent to one of the above manifolds \(M\not\cong B^{13}\).
\end{mainthmcont}

Results of Wilking suggest that, in this form, the Main Theorem covers a fairly general class of manifolds with positive sectional curvature:  any \(n\)-dimensional simply connected manifold \(M^n\) with positive sectional curvature is either isometric to a homogeneous space with positive sectional curvature or tangentially homotopy equivalent to a CROSS, provided its isometry group \(I(M)\) satisfies one of the following \cite{Wilking:positively}:
\begin{compactitem}
\item the dimension of \(I(M)\) is at least \(2n-6\), or
\item the cohomogeneity of the \(I(M)\)-action is \(\ell\geq 1\) and \(n\geq 18(\ell+1)^2\).
\end{compactitem}
On the other hand, the general form of the Main Theorem also applies to manifolds that do not admit a metric of non-negative sectional curvature:

\begin{example} \label{EXAM:EXOTIC-SPHERE}
All homotopy spheres are tangentially homotopy equivalent since their tangent bundles are stably trivial \cite{KM:Groups}*{Theorem 3.1}.
But there are homotopy spheres  \(\Sigma^m\) of dimensions \(m = 8n+1\) and \(8n+2\) that do not admit a metric of non-negative sectional curvature.  (See \cite{DT:note} for these and other pairs of homotopy equivalent manifolds of which only one admits a metric of non-negative curvature.)
\end{example}

We can use this example to illustrate the need for stabilization:
\begin{example}\label{EXAM:EXOTIC}
Let \(\Sigma^m\) be one of the homotopy spheres of the previous example that does not admit a metric of non-negative sectional curvature.
Then the total space \(\Sigma^m\times\R\) of the trivial bundle likewise admits no metric of non-negative sectional curvature (see \Cref{L:soulsCodim1}).
Our Main Theorem implies that, nonetheless, some stabilization \(\Sigma^m\times\R^k\) does admit such a metric.
In this particular case, this can also be seen directly as follows:  by the ``Work Horse Theorem'' of \cite{TW:moduli} (see \cref{sec:WorkHorse} below),
the product \(\Sigma^m\times\RR^{m+1}\) is diffeomorphic to \(\SS^m\times\RR^{m+1}\), on which a metric of non-negative curvature is provided by the product metric.
It would be interesting to
find the minimum \(2\leq k \leq m+1\) such that \(\Sigma^m\times\RR^{k}\) admits such a metric.
\end{example}

In the case of \(B^{13}\) our results are inconclusive.
We can, however, make precise to what extent the general strategy used in this article fails for \(B^{13}\): see \cref{THM:B13} below.

\subsubsection{Outline of the proof and related results}
We proceed to sketch our proof of the Main Theorem and to describe some positive results for other families of homogeneous spaces that we have obtained along the way.
First, we note that the second part of the theorem follows from the first:

\begin{prop}\label{SCST for tangentially homotopy equivalent manifolds}
If the SCSQ has a positive answer for all real (resp.\ all complex) vector bundles over \(M\), then it also has a positive answer for all real (resp.\ all complex) vector bundles over any tangentially homotopy equivalent manifold~\(N\).
\end{prop}
Here, the SCSQ for a complex vector bundle is simply the SCSQ for the underlying real vector bundle.
The proof of \cref{SCST for tangentially homotopy equivalent manifolds} given in \cref{sec:WorkHorse} relies on \namecref{Work horse} already mentioned: the main point is that  the total space of a vector bundle over \(N\) of sufficiently high rank is diffeomorphic to the total space of some vector bundle over~\(M\).

Given \cref{SCST for tangentially homotopy equivalent manifolds} and the results of \cite{Gonzalez:nonnegative} on CROSSes, it remains to prove the Main Theorem for
\(W^6\), \(W^{12}\), \(W^{24}\), \(B^7\) and \(W^{7}_{p,q}\).
Each of these is a compact homogeneous space, \ie an orbit space \(G/H\) of a subgroup \(H\) acting on a compact Lie group \(G\).
Our basic strategy is the same as in \cite{Gonzalez:nonnegative}:  we prove that every vector bundle over one of these spaces is stably isomorphic to a homogeneous vector bundle \(G\times_H\R^{m}\) for some representation of \(H\).
Such a homogeneous vector bundle always admits a metric of non-negative sectional curvature by O'Neill's theorem on Riemannian submersions.
In fact, for each homogeneous metric \(\langle,\rangle\) of non-negative curvature on \(G/H\), there is a metric on \(G\times_H\R^{m}\) with non-negative curvature and soul isometric to \((G/H,\langle,\rangle)\).

\begin{rem}
Given a stably homogeneous vector bundle \(E\) over \(G/H\),  one can estimate the minimal \(k\) needed to obtain non-negative curvature on \(E\times\RR^k\) as follows: Let \(\rho\) be a representation of \(H\) of \(\rank \rho \geq \dim G/H +1\) and minimal among all representations such that the associated homogeneous bundle is stably isomorphic to \(E\).
Then the minimum \(k\) needed satisfies \(k \leq \rank\rho - \rank E\).
See \cite{Gonzalez:nonnegative} for concrete bounds in the case \(G/H=\SS^n\).
\end{rem}

The K-rings of real or complex vector bundles over a manifold \(M\), denoted \(KO(M)\) and \(K(M)\), respectively, provide the natural framework to study stable classes of vector bundles.
Let \(RO(H)\) and \(R(H)\) likewise denote the real and complex representation rings of a Lie group \(H\).
Sending a real or complex representation of \(H\) to the corresponding real or complex homogeneous vector bundle over \(M\) defines ring homomorphisms:
\begin{align*}
\alpha_O\colon & RO(H) \to KO(G/H)\\
\alpha\colon & R(H) \to K(G/H)
\end{align*}
As observed in \cite{Gonzalez:nonnegative}, every real vector bundle over \(G/H\) can be stabilized to a homogeneous real vector bundle if and only if \(\alpha_O\) is surjective.
The same equivalence holds for complex vector bundles and \(\alpha\).
What is more, the complex version \(\alpha\colon R(H)\to K(G/H)\) is indeed surjective in many  situations.
In particular, it is well-known that \(\alpha\) is surjective whenever \(G\) is compact simply connected and \(H\subset G\) is a closed connected subgroup of maximal rank, \ie when \(\rank H = \rank G\).
(The rank of a Lie group is the dimension of a maximal torus.)  As we explain in \cref{Improvement of Pittie's} below, this is true more generally, though less well-known, for subgroups \(H\) with \(\rank G - \rank H \leq 1\).
All compact homogeneous positively curved spaces \(G/H\) satisfy this inequality, with equality if and only if the dimension of the space is odd.
As an immediate corollary, we obtain the following result for complex vector bundles over such spaces.

\begin{thm}\label{THM: complex SCST for homogeneous}
Let \(G\) be a compact connected Lie group with \(\pi_1(G)\) torsion-free, and let \(H\subset G\) be a closed connected subgroup such that \(\rank G - \rank H \leq 1\).
Then the SCSQ has a positive answer for all complex vector bundles over any compact
manifold tangentially homotopy equivalent to \(G/H\).
\end{thm}

The homomorphism \(\alpha_O\), on the other hand, is surjective only in very special situations.
For example, the results of the second-named author in \cite{Z:KO-rings} imply that \(\alpha_O\) is almost never surjective when \(G/H\) is a full flag manifold, \ie when \(H\subset G\) is a maximal torus~\(T\).
The few positive cases are:

\begin{thm}\label{THM: SCST for flag manifolds}
The SCSQ has a positive answer for all real vector bundles over the full flag manifolds \(SU(m)/T\) with \(m\in\{2,3,4,5,7\}\), over \(\Spin(7)/T\), over \(G_2/T\) and over the products of these specified in \cref{surjectivity-of-alphaO-for-full-flags} below.
\end{thm}

It therefore came as a surprise that, nonetheless, \(\alpha_O\) \emph{is} surjective for all compact simply connected manifolds admitting a homogeneous metric of positive sectional curvature, except for \(B^{13}\).
For any given homogeneous space \(M\), the surjectivity of \(\alpha_O\)  may of course depend on the precise presentation of \(M\) as \(G/H\), but for all spaces considered here, it turns out that the ``standard'' presentation with simply connected \(G\) works.
We will in fact show that \(\alpha_O\) is surjective for all homogeneous spaces in certain families that include the positively curved ones, except for \(B^{13}\).
In particular, in each of our positive cases \(W^6,W^{12},W^{24},B^7,W_{p,q}^{7}\), the surjectivity of \(\alpha_O\) fits into
one of the following two more general results.

Firstly, the spaces \(W^6\), \(B^7\) and \(W_{p,q}^{7}\) are covered by a result in low dimensions:  \(\alpha_O\) is surjective for any simply connected compact homogeneous space \(G/H\) of dimension at most seven, at least if we choose, as we may, a presentation with simply connected \(G\) (\cref{P: surjectivity alpha_O dim 7}).
The dimensional bound is perhaps not surprising---real K-theory generally becomes more complicated starting from dimension eight---and indeed,
for \(\SS^2\times \SS^2\times\SS^2\times \SS^2\) the result is false (\cref{optimility of P: surjectivity alpha_O dim 7}).
In any case, as a consequence we obtain:
\begin{thm}\label{THM: SISC for homogeneous of dim 7}
The SCSQ has a positive answer for all real vector bundles over any simply connected compact homogeneous space of dimension at most seven.
\end{thm}

Secondly, the Wallach manifolds \(W^{12}=Sp(3)/Sp(1)^3\) and \(W^{24}=F_4 /\Spin(8)\) are covered by the following result (see \cref{PROP: surjectivity alpha_O for rkG rk H conjugation}):
\begin{thm}\label{THM: SCSQ for rkG rkH and conjugation}
Let \(G\) be a compact connected Lie group with \(\pi_1(G)\) torsion-free, and let \(H\subset G\) be a closed connected subgroup with \(\rank H = \rank G\), and such that complex conjugation acts on the complex representation ring \(R(H)\) as the identity map.
Then the SCSQ has a positive answer for all real vector bundles over \(G/H\).
\end{thm}

The maximal rank closed subgroups of compact Lie groups are described in \cite{BDe:Les}.
Conjugation acts trivially on \(R(H)\) for all compact connected Lie groups \(H\) that have no circle factor and no simple factors of types \(A_n\) with \(n\geq 2\), \(D_{2n+1}\) with \(n\geq 4\) or \(E_6\) (\cref{rem:t-trivial-on-R}).

\subsubsection{The SCSQ for products}
The class of spaces satisfying \cref{THM: SCSQ for rkG rkH and conjugation} is closed under taking products.
In particular, it includes arbitrary products of \(4n\)-dimensional spheres, quaternionic projective spaces and the Cayley plane (for the usual homogeneous presentations).
It turns out that, in good situations, surjectivity of \(\alpha_O\) is also preserved when taking products with \(\CP^{2n}\) (\cref{surjectivity of alpha_O for products}).
Altogether, this gives the following result.

\begin{thm}\label{THM: SCST for products of 4n CROSSes}
Let \(M\) be a finite product of simply connected CROSSes whose dimensions are multiples of four.
Then the SCSQ has a positive answer for all real vector bundles over~\(M\).
\end{thm}

Besides the products of spheres included in \cref{THM: SCST for products of 4n CROSSes}, our \(K\)-theoretic computations yield positive results for other products of spheres.

\begin{thm}\label{THM: SCST for products of spheres}
\newcommand{\evennr}{\text{\textnormal{\small even}}}
\newcommand{\oddnr}{\text{\textnormal{\small odd}}}
Consider a product of spheres \(S=\SS^{n_1}\times\SS^{n_2} \times\dots\times\SS^{n_\ell}\)
Assume that the unordered tuple \(\{n_1,\dots ,n_\ell \}\) is congruent modulo eight to one of the following, where \(\evennr\), \(\evennr'\) denote arbitrary even numbers and \(\oddnr\) denotes an arbitrary odd number:
\begin{compactitem}[\(\ell\in\NN:\)]
\item[\(\ell=1:\)] \(\{n\}\) arbitrary
\item[\(\ell=2:\)]
  \(\{3,3\}\), \(\{7,7\}\), \(\{\evennr,\oddnr\}\) or \(\{\evennr,\evennr'\}\not\equiv \{2,6\}\)
\item[\(\ell=3:\)]
  \(\{2,2,2\},\{6,6,6\},\{7,7,7\}\)
  or \mbox{\(\{\evennr,\evennr',\oddnr\}\not\in \left\{\{2,6,\oddnr\},\{4,4,\oddnr\}\right\}\)}
\item[\(\ell=4:\)]
  \(\{2,2,2,\oddnr\}\) or \(\{6,6,6,\oddnr\}\)
\item[\(\ell\in\NN:\)]
  \(\{\evennr,n_2,\dots, n_{\ell}\}\), where \(n_i\equiv 0\textrm{ or }4\) for all \(i\)
\item[\(\ell\in\NN:\)]
  \(A\cup \{n_k,\dots, n_{\ell}\}\), where \(n_i\equiv 0\) for all \(i\) and \(A\) is any of the tuples above.
\end{compactitem}
Then the SCSQ has a positive answer for all real vector bundles over~\(S\).
\end{thm}
Finally, all of the above results that rely on the surjectivity of \(\alpha_O\) can be improved using Bott periodicity, in the sense that we can allow products with an arbitrary number of spheres whose dimensions are multiples of eight.  In combination with \cref{SCST for tangentially homotopy equivalent manifolds} we obtain:

\begin{thm}\label{THM: SCST for products with S8}
\renewcommand{\creflastconjunction}{ or~}
Let \(M\) be any of the homogeneous spaces satisfying the hypotheses in
\cref{THM: SISC for homogeneous of dim 7,,%
  THM: SCSQ for rkG rkH and conjugation,,%
  THM: SCST for flag manifolds,,%
  THM: SCST for products of 4n CROSSes,,%
  THM: SCST for products of spheres},
and let \(S\) be a product of arbitrarily many spheres whose dimensions are multiples of eight.
The SCSQ has a positive answer for all real vector bundles over any compact manifold tangentially homotopy equivalent to \(M\times S\).
\end{thm}

\subsubsection{The Berger space \texorpdfstring{$B^{13}$}{B13}}

For the standard presentation of the Berger space \(B^{13}\) as \(SU(5)/Sp(2)\times_{\Z_2}S^1\), in turns out that \(\alpha_O\) is \emph{not} surjective. As is shown in \cite{FloritZiller:Bazaikin}*{Proof of Theorem A}, any compact homogeneous space homotopy equivalent to \(B^{13}\) is already diffeomorphic to \(B^{13}\), and the standard presentation is its only presentation as an orbit space of a simply connected Lie group. It follows that \(\alpha_O\) cannot be surjective for \emph{any} presentation of \(B^{13}\) as a homogeneous space (\cref{suffices-to-consider-simply-connected-G}).
In other words, the SCSQ over \(B^{13}\) cannot be addressed by the techniques used here.
However, surjectivity of \(\alpha_O\) fails only by a small margin:

\begin{thm}\label{THM:B13}
  The map \(\alpha_O\) is not surjective for any presentation of \(B^{13}\) as a homogeneous space.
  Rather, its image is contained in a fixed subring \(S\subset KO(B^{13})\) of index two.
  For the standard presentation of  \(B^{13}\) as \(SU(5)/Sp(2)\times_{\Z_2}S^1\), the image of \(\alpha_O\) is equal to \(S\).

  Let \(\reduced{KO}(B^{13})\subset KO(B^{13})\) and \(\reduced{S}\subset S\) denote the ideals consisting of all virtual vector bundles of rank zero.
  Then \(\reduced{KO}(B^{13})\) splits, as a non-unitary ring, as a product
  \[
    \reduced{KO}(B^{13}) \cong S \times \ZZ_2 w
  \]
  with \(w^2 = 0\).
  The generator \(w\) is the unique nonzero class in \(\reduced{KO}(B^{13})\) with trivial Pontryagin classes.
\end{thm}

In particular, there exists a real vector bundle \(E\) (of rank \(\leq 13\)) over \(B^{13}\) such that \(E\times\R^k\) is not a homogeneous vector bundle for any integer \(k\).
There are, in fact, countably infinitely many such bundles -- all representatives of all elements in the \(w\)-coset of \(S\) -- but there is a unique stable class of vector bundles with the additional property that the Pontryagin classes are trivial.

\subsubsection{Acknowledgements}
The second author is grateful to Wilhelm Singhof for inspiration regarding the projective dimension of representation rings (\cref{cor:proj-dim}).

%%% Local Variables:
%%% mode: latex
%%% TeX-master: "main.tex"
%%% End:

%% file: KTheory.tex
In this section we review mostly well-known facts on the real and complex \(K\)-theory of a manifold. Classical references include \cites{Atiyah: K-theory} and \cite{Husemoller}.  We place particular emphasis on relations between the real and the complex theory.
In particular, we discuss cohomological conditions that ensure the surjectivity of the realification map from complex to real K-theory for manifolds of dimension at most seven.

\subsection{K-theory classifies vector bundles}\label{SS:K-theory rings}

Throughout this section, \(M\) will denote a compact Hausdorff space.
We denote by \(\Vect_\F(M)\) the set of isomorphism classes of (continuous) \(\F\)-vector bundles over \(M\), where \(\F\) is either \(\R\), \(\C\) or \(\HH\).  (In the last case, each fiber is considered as a right module over \(\HH\).)  Over a manifold, any continuous vector bundle can be equipped with an essentially unique smooth structure, \ie \(\Vect_\F(M)\) agrees with the set of isomorphism classes of \emph{smooth} \(\F\)-bundles.
The Whitney sum \(\oplus\) endows \(\Vect_\F(M)\) with a semigroup structure, and we define \(K_\F(M)\) as its group completion. The elements of \(K_\F(M)\) can be described as formal differences of elements in \(\Vect_\F(M)\).  As an example, for a point we have \(K_\F(pt) \cong \ZZ\).  As usual in the literature, we denote \(K_\R\), \(K_\C\) and \(K_\HH\) by \(KO\), \(K\) and \(KSp\), respectively.

\subsubsection{Products}
When \(\F=\R\) or \(\C\), the \(\F\)-tensor product gives \(\Vect_\F(M)\) a semiring structure. Thus \(K_\F(M)\) is the ring completion of \(\Vect_\F(M)\) and \(K(M)\) and \(KO(M)\) are referred to as the complex and real K-rings  of \(M\), respectively.
Moreover, recall that there is a bijective correspondence between \(\Vect_\RR(M)\) (resp. \(\Vect_\HH(M)\)) and the set of isomorphism classes of \(\CC\)-vector bundles endowed with a conjugate-linear automorphism \(J\) such that \(J^2 = \id\) (resp. \(J^2 = -\id\)) \cite{Atiyah: K-theory}*{\S\,1.5}.
In this way, the \(\CC\)-tensor product induces product maps \(KO(M)\otimes KSp(M)\to KSp(M)\) and \(KSp(M)\otimes KSp(M)\to KO(M)\) giving the direct sum \(KO(M)\oplus KSp(M)\) the structure of a \(\ZZ_2\)-graded ring.
Similarly, the \(i\)-th exterior power of vector spaces induces the corresponding exterior power operations \(\lambda^i\) on vector bundles, making \(K(M)\) and \(KO(M)\) \(\lambda\)-rings and \(KO(M)\oplus KSp(M)\) a \(\ZZ_2\)-graded \(\lambda\)-ring \citelist{\cite{Bott:quelques}\cite{Allard}}.
For two spaces \(M_1\) and \(M_2\) we also have \emph{external} products defined in the usual way in terms of the (internal) products discussed above and the canonical projections \(\pi_i\colon M_1\times M_2\to M_i\):
\[
\begin{array}{ccc}
\mu\colon K_{\F_1}(M_1)\otimes K_{\F_2}(M_2) & \to & K_{\F_{12}}(M_1\times M_2) \\
\alpha\otimes\beta & \mapsto & \pi_1^*(\alpha)\cdot \pi_2^*(\beta)
\end{array}
\]
Here, \((\F_1,\F_2,\F_{12})\) is one of the triples \((\CC,\CC,\CC)\), \((\RR,\RR,\RR)\), \((\RR,\HH,\HH)\), \((\HH,\RR,\HH)\) or \((\HH,\HH,\RR)\).  Bott periodicity implies that the external product for \(K\)- and \(KO\)-theory is an isomorphism when one of the two spaces is an even-dimensional sphere or a sphere of dimension a multiple of eight, respectively \citelist{\cite{Atiyah:K-theory}*{Cor.~2.2.3}\cite{Husemoller}*{Thm~11.1.2}}:
\begin{equation}\label{eq:external-product-iso}
  \begin{alignedat}{8}
    K(M)&\otimes K(\SS^{2k})&&\xrightarrow{\;\cong\;} & K(&M\times \SS^{2k})
    \\
    KO(M)&\otimes KO(\SS^{8k}) &&\xrightarrow{\;\cong\;}\; & KO(&M\times \SS^{8k})
  \end{alignedat}
\end{equation}
(Most references only state the case \(k=1\).  The more general statement follows from \cite{Husemoller}*{Prop.~11.1.1}.)
See also \cite{Atiyah:Kunneth} for a general Künneth formula for complex K-theory.

\subsubsection{Reduced K-groups}
For a pointed space \((M,pt)\), the reduced \(K_\F\)-theory
is defined as the kernel of the restriction map \(K_\F(M)\to K_\F(pt)\).  When \(M\) is connected, this subgroup
is independent of the chosen basepoint and we denote it by \(\reduced{K}_\F(M)\).  It may in this case be described explicitly as follows.  Denote by \(n_\F\) the trivial \(\F\)-vector bundle over \(M\) of rank \(n\in\N\). Then
\[
\reduced{K}_\F(M) = \{ E-n_\F \in K_F(M) \mid \rank E = n\}.
\]
Alternatively, \(\reduced{K}_\F(M)\) can be described as the abelian group of stable classes of \(\F\)-vector bundles over \(M\).  Two vector bundles \(E,F\in\Vect_\F(M)\) are stably equivalent if there exist trivial bundles \(n_\F,m_\F\) such that \(E\oplus n_\F\) is isomorphic to \(F\oplus m_\F\).
In any case, we have a natural isomorphism of abelian groups:
\begin{equation}
K_\F(M)=\Z\oplus\reduced{K}_\F(M).
\label{EQ: bracket Z}
\end{equation}
As we will ultimately be interested in the reduced group \(\reduced{K}_\F(M)\), we will often write the constant \(\Z\)-term in \eqref{EQ: bracket Z} in parentheses, \ie
\(
K_\F(M)=[\Z ]\oplus\reduced{K}_\F(M)
\).

\subsection{K-theory as a generalized cohomology theory}
Atiyah and Hirzebruch constructed generalized cohomology theories \(K^*(-)\) and \(KO^*(-)\) on topological spaces such that \(K^0(M)\cong K(M)\) and \(KO^0(M)\cong KO(M)\) for any compact Hausdorff~\(M\).
For negative indices \(i\), the groups \(K^i(M)\) and \(KO^i(M)\) can be defined geometrically in terms of vector bundles over suspensions of~\(M\).
An equivalent way to state Bott periodicity \eqref{eq:external-product-iso} is that these cohomology theories are periodic in the sense that there are natural isomorphisms \(K^{i+2}(-)\cong K^{i}(-)\) and \(KO^{i+8}(-)\cong KO^{i}(-)\).
Their coefficient groups are as follows:
\[
\newcolumntype{L}{>{\raggedright\arraybackslash\(}p{1.1em}<{\)}}
\begin{array}{lLLLLLLLL}
{}^{i \text{ mod 8}} &{}^0 & {}^1 & {}^2 & {}^3 & {}^4 & {}^5 & {}^6 & {}^7 \\
K^{-i}(pt) & \Z & 0  & \Z & 0  & \Z & 0  & \Z & 0 \\
KO^{-i}(pt) & \Z & \Z_2  & \Z_2 & 0  & \Z & 0  & 0 & 0 \\
\end{array}
\]
Both cohomology theories are multiplicative.  Their coefficient rings may be written as
\begin{align*}
K^*(pt) &= \ZZ[\beta^{\pm 1}]  \\
 KO^*(pt) &= \ZZ[\eta,\alpha,\beta_O^{\pm 1}]/(2\eta,\eta^3,\eta\alpha,\alpha^2-4\beta_O)
\end{align*}
with \(\deg{\beta} = -2\), \(\deg{\eta} = -1\), \(\deg{\alpha} = -4\) and \(\deg{\beta_O} = -8\).
The graded rings \(K^*(M):=\bigoplus_{i\in\Z} K^{i}(M)\) and \(KO^*(M):=\bigoplus_{i\in\Z} KO^{i}(M)\)
are graded \(K^*(pt)\)- and \(KO^*(pt)\)-modules, respectively, and the periodicity isomorphisms are given by multiplication with the so-called Bott elements \(\beta\) and \(\beta_O\).
In the quaternionic case, the same construction yields a graded abelian group \(KSp^*(M):=\bigoplus_{i\in\Z} KSp^{i}(M)\). The periodicity in this case is again of period eight, as there are natural group isomorphisms:
\begin{equation}\label{eq:KSp=KO-4}
 KSp^{i}(M)\cong KO^{i-4}(M), \text{ for all \(i\in\Z\).}
\end{equation}

For connected \(M\), we define higher \emph{reduced} groups \(\reduced{K}^i(M)\) as in the case \(i=0\), so that we have natural splittings:
\begin{align*}
K^i(M)&=[K^i(pt)]\oplus\reduced{K}^i(M), & KO^i(M)&=[KO^i(pt)]\oplus\reduced{KO}^i(M).
\end{align*}

\subsection{Relations between complex and real K-theory}\label{sec:Tate}\label{sec:c,r,q,t}
We have the following natural maps between the K-groups of a space, referred to as complexification (\(\cx\), \(\cx'\)), realification (\(\rx)\), quaternionification (\(\qx\)) and conjugation (\(\tx\)):
\begin{align*}
&\cx\colon KO(M)\to K(M) & & \rx\colon K(M)\to KO(M) & &\tx\colon K(M)\to K(M)  \\
  & \cx'\colon KSp(M)\to K(M) & &\qx\colon K(M)\to KSp(M) & &
\end{align*}
They satisfy the following identities, where \(2\) denotes multiplication by two:
\begin{align*}
& \rx\circ\cx=2 & & \cx\circ\rx = \id+\tx & & \qx\circ\cx' =2  \\
& \cx'\circ\qx=\id+\tx & & \tx\circ\cx=\cx & & \rx\circ\tx=\rx \\
&\tx\circ\cx'=\cx' & & \qx\circ\tx = \qx & & \tx^2=\id
\end{align*}
The maps \(\cx\) and \(\tx\) are ring homomorphisms, while \(\cx'\), \(\rx\) and \(\qx\) are only group homomorphisms.  More precisely, \((\cx,\cx')\colon KO(M)\oplus KSp(M)\to K(M)\) is a ring homomorphism.  This gives \(K(M)\) the structure of a \(KO(M)\oplus KSp(M)\)-module.  The map \(\paircolumn{\rx}{\qx}\colon K(M)\to KO(M)\oplus KSp(M)\) is a \(KO(M)\oplus KSp(M)\)-module homomorphism with respect to this module structure on \(K(M)\).  For example, \(\rx(\cx' x \cdot y) = x\cdot \qx y\) for all \(x\in KSp(M)\), \(y\in K(M)\).

All of the above maps extend to natural transformations of cohomology theories.
Under the identification~\eqref{eq:KSp=KO-4}, \(\qx\colon K(M)\to KSp(M)\) corresponds to \(\rx\colon K^{-4}(M)\to KO^{-4}(M)\) and \(\cx'\) corresponds to \(\cx\) \cite{Bott:quelques}.  The values of \(\cx\), \(\rx\) and \(\tx\) on the coefficient rings are determined by:
\begin{align}\label{eq:c-r-t-on-coefficients}
 \tx\beta &= -\beta &&& \cx\beta_O&= \beta^4 &  \cx\alpha&=2\beta^2 &&& \rx\beta&= \eta^2 & \rx(\beta^2) &= \alpha
\end{align}
\cite{Bousfield:K-local}*{\S\,1.1}.
The realification and complexification maps in different degrees can be assembled into a periodic exact sequence known as the \emph{Bott sequence}, of the following form:
\begin{equation}\label{EQN: Bott sequence}
\newcommand{\lbl}[1]{\parbox{2em}{\centering\(#1\)}}
\cdots\to KO^{i}(M) \xrightarrow{\lbl{\cx}} K^{i}(M)\xrightarrow{\lbl{\rx\beta^{-1}}} KO^{i+2}(M)\xrightarrow{\lbl{\eta}} KO^{i+1}(M)\to\cdots
\end{equation}

\subsubsection{A lemma of Bousfield}
We can view the conjugation \(\tx\colon K(M)\to K(M)\) as an involution on the K-ring.
In general, given a ring \(A=(A,\tx)\) with an involution, \ie with a ring automorphism \(\tx\) such that \(\tx^2=\id\), one can define the associated Tate ring as the \(\ZZ_2\)-graded ring
\(
  h^*(A,\tx) := h^+(A,\tx) \oplus h^-(A,\tx)
\)
with homogeneous components:
\begin{align*}
  h^+(A,\tx) &= \frac{\ker(\id-\tx)}{\im(\id+\tx)}\\
  h^-(A,\tx) &= \frac{\ker(\id+\tx)}{\im(\id-\tx)}
\end{align*}

\begin{lem}[Bousfield]\label{Bousfield's-lemma}
For any space \(M\) with \(K^1(M)=0\), the complexification maps induce isomorphisms:
\begin{align*}
KO(M)/\rx \oplus KSp(M)/\qx &\xrightarrow[\;\cx\,+\,\cx']{\cong}
   h^+(K(M),\tx) \\
KO^2(M)/\rx \oplus KO^6(M)/\rx &\xrightarrow[\;\cx\,+\,\cx\;]{\cong}
   h^-(K(M),\tx)
\end{align*}
\end{lem}
\begin{proof}
This is essentially \cite{Bousfield:2-primary}*{Lemma~4.7}.  See also \cite{Z:KO-rings}*{\S\,1.2, \S\,2.1}.
\end{proof}

\begin{rem}\label{Bousfield's-lemma-for-RG}
Likewise, for any compact Lie group \(G\), complex conjugation defines an involution \(\tx\) on the complex representation ring \(RG\), and the complexification map from the real\slash quaternionic representation ring \(RO(G)\oplus RSp(G)\) to the complex representation ring induces a ring isomorphism:
\[
RO(G)/\rx \oplus RSp(G)/\qx \xrightarrow[\cx\,+\,\cx']{\cong}  h^+(RG,\tx)
\]
The odd Tate groups \(h^-(RG,\tx)\) vanish in this case (see section \cref{SS: homogeneous vector bundles}).

\end{rem}

\begin{lem}\label{Kunneth-for-Tate}
  Suppose \((A,\tx)\) and \((B,\tx)\) are commutative rings with involutions.  Consider the tensor product ring \(A\otimes B\) with the induced involution \(\tx\otimes\tx\).  If \(A\) is free as an abelian group, then the canonical ring homomorphism of \(\ZZ_2\)-graded rings \(h^*(A)\otimes h^*(B)\to h^*(A\otimes B)\) is an isomorphism.
\end{lem}
\begin{proof}
  As an abelian group, we can decompose \(A\) into a direct sum \(A_+\oplus A_- \oplus A' \oplus A'\) such that \(\tx\) acts on \(A_+\) as the identity, on \(A_-\) as minus the identity and on \(A'\oplus A'\) by interchanging the two factors \cite{Bousfield:K-local}*{Prop.~3.7}. It suffices to verify the claim separately for each of these three factors.
\end{proof}

\subsection{The Atiyah-Hirzebruch spectral sequences}\label{sec:AHSS}
In \cite{AH:bundles homogeneous}, Atiyah and Hirzebruch constructed a spectral sequence for calculating generalized cohomology theories like \(K^*\) and \(KO^*\) of a space from its ordinary cohomology. It is commonly known today as the Atiyah-Hirzebruch spectral sequence (AHSS).

A spectral sequence is a sequence of pages \(\{E_k^{p,q}\}_{k\in\Z}\), each consisting of a grid of abelian groups together with certain homomorphisms \(d_k^{p,q}\) between them, called differentials, that determine the next page.  See \cite{McCleary:SS} for a comprehensive introduction to this machinery.  A spectral sequence converges to some graded group \(G^*\) if there exists a limit page \(E_\infty^{p,q}\) such that each diagonal \(p+q=c\) is isomorphic to the graded group associated with some bounded exhaustive filtration on \(G^c\). We refer to this situation with the notation \(E_2^{p,q} \Rightarrow G^{p+q}\).

The AHSS for the generalized cohomology theories \(K^*(M)\) and \(KO^*(M)\) have the form:
\begin{align*}
E_2^{p,q}=H^p (M, K^q(pt))  &\Rightarrow K^{p+q}(M) \\
E_2^{p,q}=H^p (M, KO^q(pt)) &\Rightarrow KO^{p+q}(M)
\end{align*}
Alternatively, using reduced cohomology, we obtain a spectral sequence converging to the reduced K- and KO-theory of \(M\).  It differs from the unreduced AHSS only in that the first column is zero.  We will repeatedly use the following facts:
\begin{enumerate}[(a)]
\item
  The images of all differentials are torsion.
  In particular, if \(H^*(M,\Z)\) is torsion-free, then there are no nontrivial differentials in the AHSS for \(K^*(M)\), and consequently \(E_2^{p,q} = E_\infty^{p,q}\).
  (See \cite{AH:bundles homogeneous}*{\S\,2.4} for the complex case.  The real case follows by comparing the two spectral sequence for \(K^*(-)\otimes\QQ\) and \(KO^*(-)\otimes\QQ\).)
\item
  In the AHSS for \(KO^*(M)\), the differential \(d_2\colon E_2^{p,q}\to E_2^{p+2,q-1}\) is as follows:
  \[
    d_2^{p,q}=
    \begin{cases}
      Sq^2\circ\mathsf{red} & \text{ for } q\equiv 0 \mod 8  \\
      Sq^2                & \text{ for } q\equiv -1 \mod 8 \\
      0                   & \text{ otherwise, }
    \end{cases}
  \]
  Here, \(Sq^2\) denotes the second Steenrod square and \(\mathsf{red}\) is reduction of coefficients modulo two. \cite{Fujii}*{\S\,1}

\item
  The spectral sequences are multiplicative.  That is, for each page \(E_r\) of the spectral sequence, we have a pairing of bigraded abelian groups \(E_r\otimes E_r\to E_r\), and with respect to this pairing, the differential \(d_r\) satisfies the Leibniz rule \(d_r(x\cdot y) = d_r(x) \cdot y +(-1)^{p} x \cdot d_r(y)\) for \(x\in E_r^{p,q}\) and \(y\in E_r^{s,t}\).
  Up to a sign, the pairing on the \(E_2\)-page can be identified with the usual cup product \cite{Dugger:MII}*{Thm~3.4 and Rem.~3.6}.

\item \label{AHSS-facts:characteristic-classes}
The first few columns of the spectral sequences are closely related to certain characteristic classes.  Write \(K(M) = F^0 K(M) \supset F^1 K(M) \supset F^2 K(M) \supset \cdots\) for the filtration on \(K(M)\) whose associated graded ring is isomorphic to the zero diagonal of the \(E_\infty\)-page of the AHSS converging to \(K(M)\).  As all odd rows vanish, \(F^{2i-1}K(M) = F^{2i}K(M)\) for all~\(i\). Moreover, \(F^2 K(M) = \reduced K(M)\), the morphism \(F^2 K(M) \to H^2(M,\ZZ)\) coming from the spectral sequence can, up to a sign, be identified with the first Chern class, and the morphism \(F^4 K(M)\to H^4(X,\ZZ)\) can be identified with the restriction of the second Chern class, again up to a possible sign.  Similarly, for the corresponding filtration on \(KO(M)\), we have \(F^1 KO(M) = \reduced{KO}(M)\), the map \(F^1 KO(M)\to H^1(M,\ZZ_2)\) is the first Stiefel-Whitney class, and the map \(F^2 KO(M)\to H^2(M,\ZZ_2)\) is the second Stiefel-Whitney class \cite{Z:WCS}*{Lem.\,3.13}.
\end{enumerate}

Here's one quick application:
\begin{lem}\label{lem: triviality ok KO product spheres}
A product of spheres \(S=\SS^{n_1}\times\SS^{n_2}\times\dots\times\SS^{n_\ell}\) satisfies \(\reduced{KO}(S)=0\) only in the following cases:
\begin{compactitem}
\item[\(\ell=1\)] and \(n_1\equiv 3,5,6,7 \text{ modulo } 8\).
\item[\(\ell=2\)] and the unordered pair \(\{n_1,n_2\}\) is congruent to one of the pairs \(\{3,3\},\) \(\{5,6\}\), \(\{6,7\}\) or \(\{7,7\}\) modulo~\(8\).
\item[\(\ell=3\)] and the triple \(\{n_1,n_2,n_3\}\) is congruent to \(\{7,7,7\}\) modulo~\(8\).
\end{compactitem}
\end{lem}
\begin{proof}
The differentials in the AHSS for \(KO\) are trivial for any sphere.  It follows from the multiplicativity of the spectral sequence that they are also trivial for~\(S\). Thus, \(\reduced{KO}(S)\) vanishes if and only if all terms on the diagonal \(E_2^{p,-p}\) are zero for \(p\geq 1\).  As the cohomology of \(S\) is finitely generated free, this happens if and only if \(H^k(S,\Z)\) vanishes in all positive degrees \(k\equiv 0,1,2,4 \text{ mod } 8\).  This condition is equivalent to the numerical conditions stated.
\end{proof}

The structure maps \(\rx\), \(\cx\) and \(\tx\) induce morphisms of spectral sequences between the AHSS computing the K- and KO-theory of a space.  For example, \(\rx\) induces a sequence of maps \(\rx_k^{p,q}\colon E_k^{p,q}(K(M))\to E_k^{p,q}(KO(M))\) that commute with the differentials and converge to \(\rx\colon K^*(M)\to KO^*(M)\).   On the \(E_2\)-pages, these morphisms are simply the change-of-coefficient maps induced by the corresponding maps between the K- and KO-theory of a point.  Using \eqref{eq:c-r-t-on-coefficients}, we can easily describe them explicitly:

\begin{description}
\item[Conjugation] \(\tx_2^{p,q}\colon E_2^{p,q}(K(M))\to E_2^{p,q}(K(M))\)
  \begin{compactitem}
  \item [For \(q\equiv 0\text{ mod } 4\),]  \(\tx_2^{p,q}\colon H^p(M,\Z)\to H^p(M,\Z)\) is the identity.
  \item [For \(q\equiv 2\text{ mod } 4\),]  \(\tx_2^{p,q}\colon H^p(M,\Z)\to H^p(M,\Z)\) is multiplication by \(-1\).
  \item [In all other cases,] \(\tx_2^{p,q}\) vanishes.
  \end{compactitem}

\item[Complexification] \(\cx_2^{p,q}\colon E_2^{p,q}(KO(M))\to E_2^{p,q}(K(M))\)
  \begin{compactitem}
  \item [For \(q\equiv 0\text{ mod } 8\),]  \(\cx_2^{p,q}\colon H^p(M,\Z)\to H^p(M,\Z)\) is the identity.
  \item [For \(q\equiv 4\text{ mod } 8\),]  \(\cx_2^{p,q}\colon H^p(M,\Z)\to H^p(M,\Z)\) is multiplication by \(2\).
  \item [In all other cases,] \(\cx_2^{p,q}\) vanishes.
  \end{compactitem}

\item[Realification] \(\rx_2^{p,q}\colon E_2^{p,q}(K(M))\to E_2^{p,q}(KO(M))\)
  \begin{compactitem}
  \item [For \(q\equiv 0\text{ mod } 8\),]  \(\rx_2^{p,q}\colon H^p(M,\Z)\to H^p(M,\Z)\) is multiplication by~\(2\).
  \item [For \(q\equiv 4\text{ mod } 8\),]  \(\rx_2^{p,q}\colon H^p(M,\Z)\to H^p(M,\Z)\) is the identity.
  \item [For \(q\equiv 6\text{ mod } 8\),]  \(\rx_2^{p,q}\colon H^p(M,\Z)\to H^p(M,\Z_2)\) is reduction of coefficients.
  \item [In all other cases,] \(\rx_2^{p,q}\) vanishes.
  \end{compactitem}
\end{description}

For example, the description of \(\rx_2\) implies:
\begin{lem}\label{lem: realification argument}
  Suppose all differentials in the AHSS's affecting elements in the diagonals computing \(K^0(M)\) and \(KO^0(M)\) vanish. If %\COMMENT[MZ]{This is \emph{not} an iff.}
  \begin{compactitem}
  \item \(H^k(M)=0\) for positive \(k\equiv 0\text{ mod } 8\), and
  \item reduction of coefficients \(H^k(M,\Z)\to H^k(M,\Z_2)\) is surjective for \(k\equiv 2\text{ mod } 8\)
  \end{compactitem}
  then the map \(\rx\colon  \reduced K(M)\to \reduced{KO}(M)\) is surjective.
\end{lem}
A similar argument to that in the proof of \cref{lem: triviality ok KO product spheres} yields the following corollary:
\begin{cor}\label{lem: surjectivity of realif product spheres}
Consider a product of spheres \(S=\SS^{n_1}\times\SS^{n_2}\times\dots\times\SS^{n_\ell}\). If, among all possible sums \( n_{i_1}+\dots +n_{i_k}\) (\( 1\leq i_1 < \dots < i_k \leq \ell\)), only the empty sum is congruent to zero modulo eight, then the realification map \(\rx\colon  \reduced{K}(S)\to \reduced{KO}(S)\) is surjective.
\end{cor}

Observe that of course the products in \cref{lem: triviality ok KO product spheres} are included in \cref{lem: surjectivity of realif product spheres}.

\subsection{Manifolds of dimension at most seven}
Here, we apply the preceding discussion to simply connected compact manifolds \(M^n\) of small dimensions.  By the reduced realification, we will mean the restriction of the realification map to reduced K-groups: \(\rx\colon \reduced K(M)\to \reduced{KO}(M)\).

\begin{prop}\label{P: description of manifolds leq 7}
Let \(M^n\) be a simply connected compact manifold of dimension \(n\leq 7\).
If \(H^{n-2}(M,\Z)\) contains no two-torsion, then the reduced realification is surjective.
If \(H^5(M,\Z)\) is torsion-free, then the absence of two-torsion in \(H^{n-2}(M,\Z)\) is \emph{equivalent} to the surjectivity of the reduced realification.
\end{prop}
The proof below will moreover show that, when both cohomological assumptions are satisfied, the group \(\reduced{KO}(M)\) can be described as an extension of \(H^2(M,\Z_2)\) by \(H^4(M,\Z)\).  If \(H^4(M,\Z)\) is a torsion group of odd order, this extension splits, so
  \begin{equation}
    KO(M)\cong[\Z]\oplus H^2(M,\Z_2)\oplus H^4(M,\Z).
  \end{equation}

\begin{examples}\Cref{P: description of manifolds leq 7} easily implies surjectivity of the reduced realification in the following cases:\mbox{}
\begin{compactitem}
\item
  \(M\) compact simply connected of \(\dim M\leq 4 \).
\item
  \(M\) compact simply connected and \emph{homogeneous} of \(\dim M\leq 7\), except for two cases (see \cref{L: realification dim 7}).
\item
  \(M\) is an Eschenburg biquotient with positive curvature \cite{E:new}*{Proposition 36}.
  % \(M\) is a type~\(r\) manifold in the sense of \cite{Kruggel:Ahomotopy}, for some \(r\in \NN\). %\COMMENT[MZ]{We'll assume \(r>0\).}
  % This class of simply connected compact seven-dimensional manifolds includes all Eschen\-burg biquotients  and many of the candidate spaces to admit a cohomogeneity one metric of positive curvature. See \cite{GWZ:cohomo 1} for a classification of the candidates and \cite{EU:candidates} for a description of their topology. The reduced KO-theory \(\reduced{KO}(M)\) of a type~\(r\) manifold is a finite group of order~\(2r\).\COMMENT[DG]{I would completely delete this comment (and of course keep it for us)}
  %
  %\LONGCOMMENT[DG]{The cohomology of these simply connected compact \(7\)-folds is of the form:
  %  \[
  %    H^k(M,\Z) =
  %    \begin{cases}
  %      \Z & \text{ if \(k=0,2,5,7\)}       \\
  %      \Z_r  & \text{ if \(k=4\), for some integer \(r\)}
  %      \\
  %      0 & \text{otherwise}
  %    \end{cases}
  %  \]
  %}
  %
  % \LONGCOMMENT[MZ]{If I understand correctly, Kruggel also imposes conditions on the ring structure of \(H^*(M)\).  We don't need this ring structure here, but we should be careful not to be give the impression that they are defined simply by an additive description of their cohomology.}
  %
\item
  \(M\) is the total space of an \(\SS^3\)-bundle over \(\SS^4\). Grove and Ziller proved in \cite{GroveZiller:Milnor} that they all admit non-negatively curved metrics.
  %
  % \LONGCOMMENT[DG]{These are classfied by \(\pi_3(SO(4))=\Z\oplus\Z\), and with the notation \(M_{m,n}\) for \(m,n\in\Z\) as in \cite{CE:classification}, their cohomology is:
  % \[
  %   H^k(M_{m,n},\Z) =
  %   \begin{cases}
  %     \Z & \text{ if \(k=0,7\)}                                                                        \\
  %     \Z_n  & \text{ if \(k=4\)}
  %     \\
  %     0 & \text{otherwise}
  %   \end{cases}
  % \]
  % Here \(\reduced{KO}(M_{m,n})=\Z_n\). }
  %
  % \LONGCOMMENT[MZ]{If I understand correctly, \(S^3\)-bundles over \(S^4\) \emph{with structure group \(SO(4)\)} are classified by \(\pi_3(SO(4))\).  In contrast, principle \(S^3\)-bundles over \(S^4\) are classified by \(\pi_4(BS^3)=\pi_3(S^3) = \ZZ\).  So if we say something about the classification, we need to be more precise what kind of classification we mean.  On the other hand, I don't think we need the classification.  The Serre spectral sequence implies that the cohomology of any such bundle is as you write.  The only exception is the trivial bundle \(S^3\times S^4\), but here our proposition also applies. }
  %
\end{compactitem}
\end{examples}

\begin{proof}[Proof of \cref{P: description of manifolds leq 7}]
Let \(H^i := H^i(M,\Z)\) and \(h^i := H^i(M,\Z_2)\).
We note first that for a simply connected compact \(n\)-manifold \(M\), our hypothesis that \(H^{n-2}\) contains no two-torsion is equivalent to the assumption that the reduction map \(H^2\to h^2\) is surjective.  Indeed, as \(H_1\) vanishes, the Universal Coefficient Theorem identifies \(H^2\) and \(h^2\) with \(\Hom(H_2,\Z)\) and \(\Hom(H_2,\Z_2)\), respectively. The long exact sequence obtained by applying \(\Hom(H_2,-)\) to the short exact sequence \(\Z\rightarrowtail \Z \twoheadrightarrow \Z_2\) therefore yields an isomorphism between the cokernel of the reduction map \(H^2\to h^2\) and the kernel of multiplication by two on \(\Ext(H_2,\Z)\).  As \(M\) is a  compact manifold, \(H_2\) is finitely generated, and \(\Ext(H_2,\Z)\) can be identified with the torsion subgroup of \(H_2\).  Moreover, by Poincaré duality, \(H_2\cong H^{n-2}\).  So altogether the cokernel of \(H^2\to h^2\) is isomorphic to the kernel of multiplication by two on \(H^{n-2}\).

Now assume in addition that \(n\leq 7\), and consider the realification map from the AHSS computing \(\reduced{K}(M)\) to the AHSS computing \(\reduced{KO}(M)\).  The surjectivity of \(H^2\to h^2\) implies that, on the \(E_2\)-page, realification is surjective on the zero diagonal computing \(\reduced K(M)\) and \(\reduced{KO}(M)\).

The only possibly nontrivial differentials on these two diagonals are the two differentials \(d_3\) defined on \(E_{2}^{2,-2}(K) = H^2\) and \(E_2^{2,-2}(KO) = h^2\), respectively.  Both take values in \(H^5\) and thus vanish when \(H^5\) is torsion-free.  In that case, \(\reduced{K}(M)\) is an extension of \(H^2\) by \(H^4\), \(\reduced{KO}(M)\) is an extension of \(h^2\) by \(H^4\), and the reduced realification \(\rx\) defines a morphism of extensions that restricts to an isomorphism on the subgroups identified with \(H^4\) (see above \cref{lem: realification argument}).  So \(\rx\) is surjective if and only if \(H^2\to h^2\) is.

However, even when the differentials \(d_3\) are nontrivial, the realification remains surjective on the zero diagonal of the \(E_3\)-page.  Indeed, we have the following commutative diagram:
\[\xymatrix@R=12pt@C=12pt{
    0 \ar[r] & E_3^{2,-2}(\reduced{K}(M)) \ar[r]\ar[d]^{\rx_3^{2,-2}} & H^2 = H^2(M,K^{-2}(pt))\ar[d]^{\rx_2^{2,-2}}\ar[r]^-{d_3} & H^5 = H^5(M,K^{-4}(pt))\quad\ar[d]^{\rx_2^{5,-4}} \\
    0 \ar[r] & E_3^{2,-2}(\reduced{KO}(M)) \ar[r]                   & h^2 = H^2(M,KO^{-2}(pt)) \ar[r]^-{d_3}                       & H^5 = H^5(M,KO^{-4}(pt))
  }\]
We are assuming that the second vertical arrow is surjective, and we know that the third vertical arrow is an isomorphism (again, see above \cref{lem: realification argument}). So it follows that the first vertical arrow is also surjective.
Thus, again, realification is surjective on the zero diagonal of the \(E_\infty\)-pages, and hence as a morphism \(\reduced{K}(M)\to\reduced{KO}(M)\).
\end{proof}

%%% Local Variables:
%%% mode: latex
%%% TeX-master: "main"
%%% End:

%% file: Homogeneous.tex
We outline the classical construction of a vector bundle over a homogeneous space from a representation.
Moreover, we recall and extend existing results on how many stable classes of vector bundles arise in this way.
These considerations will be key to the proof of the Main Theorem.

\subsection{Homogeneous spaces}\label{sec:homogeneous spaces}
A homogeneous space is a manifold \(M\) that admits a smooth transitive action by a Lie group~\(G\).  It may be identified with the orbit space \(G/H\) for a closed subgroup \(H\subset G\), the stabilizer of a point under the given action.
In general, there can be many different groups \(G\) acting transitively on \(M\), so the presentation of \(M\) as \(G/H\) is not unique. However, the difference \(\rank G-\rank H\) is a homotopy invariant of \(M\) (see \cite{GOV:Lie}*{II\,\S\,2.2, remark following Cor.~2}).
Moreover, when \(G/H\) is compact and simply connected, we may assume the same for~\(G\):
\begin{lem}\label{presentations-with-simply-connected-G}
  Given any presentation of a simply connected compact homogeneous space as \(G/H\), there is a simply connected compact Lie group \(G'\), a closed connected subgroup \(H'\subset G\) and a homomorphism of Lie groups \(f\colon G'\to G\) with \(f(H')\subset H\) inducing a diffeomorphism \(G'/H'\cong G/H\).
\end{lem}
\begin{proof}
Suppose \(G/H\) is some presentation of \(M\).
Let \(H_1\subset H\) and \(G_1\subset G\) be the identity components.
As \(M\) is simply connected, we find that \(H/H_1\to G/G_1\) is an isomorphism, and hence that \(G/H\) is diffeomorphic to \(G_1/H_1\).
Next, by Montgomery's Theorem \cite{GOV:Lie}*{II\,S\,2.3, Cor.~3}, we have \(G_1/H_1\cong G^c/H^c\) for maximal compact subgroups \(H^c\subset H_1\) and \(G^c\subset G_1\).
Let \(G'\) denote the compact factor of the universal cover of \(G^c\), and let \(H'\) denote the preimage of \(H\) under the canonical homomorphism \(j\colon G'\to G^c\).
This homomorphism induces a diffeomorphism \(G'/H'\cong G^c/H^c\).
\end{proof}
From now on, all homogeneous spaces and all Lie groups considered in this article will be compact.

\subsection{Representations and homogeneous bundles}\label{SS: homogeneous vector bundles}
As in \cref{S:K-theory}, let \(\F\) denote either \(\R\), \(\C\) or \(\HH\).
Given a compact Lie group \(G\) and a subgroup \(i\colon H\to G\) as above, a representation \(\rho\) of \(H\) on \(\F^m\) defines a right \(H\)-action on \(G\times\F^m\) by the rule \( ((g,v),h)\mapsto (gh,\rho(h^{-1})v)\). The orbit space
\[
E_\rho := G\times_H\F^m := (G\times \F^m)/H
\]
is the total space of  an \(\F^m\)-vector bundle over \(G/H\). Vector bundles constructed in this way are called homogeneous.

\begin{example}
  The tangent bundle of a homogeneous space \(G/H\) is always a homogeneous vector bundle \cite{GOV:Lie}*{II\,\S\,3.3}. %p.\,132
  The representation of \(H\) inducing the tangent bundle is called the isotropy representation for the presentation \(G/H\).
\end{example}

The real and complex representation rings \(RO(H)\) and \(R(H)=RH\), and the quaternionic representation group \(RSp(H)\), are constructed in a manner completely analogous to the construction of K-groups of a space: one starts with the set \(\Rep_\F(H)\) of isomorphism classes of finite-dimensional \(\F\)-representations of \(H\), endows it with a semigroup structure via the direct sum, and passes to the group completion \(R_\F(H)\).  Unsurprisingly, the properties of \(R(H)\), \(RO(H)\) and \(RSp(H)\) are also analogous to those described in \cref{SS:K-theory rings} for \(K(M)\), \(KO(M)\) and \(KSp(M)\); see \cite{Adams:LieLectures}*{Ch.\,3} for details.  In particular, \(R(H)\) and \(RO(H)\) are rings via the tensor product, and \(RO(H)\oplus RSp(H)\) is a \(\ZZ_2\)-graded ring. However, in some ways representation rings are simpler than K-rings:
\begin{itemize}
\item
  The representation ring \(RH\) is always free as an abelian group.  A basis is given by the complex irreducible representations of~\(H\).
\item
  The complexification maps \(\cx\colon RO(H)\to R(H)\) and \(\cx'\colon RSp(H)\to R(H)\) are always injective \cite{Adams:LieLectures}*{Prop.~3.22}.
\item
  The negative Tate cohomology \(h^-(RH,\tx)\) always vanishes, as is immediate from the following description of \(\tx\) in terms of the given basis:
  Let us say that an irreducible complex representation is of real or quaternionic type if it is the complexification of a real or a quaternionic representation, respectively; otherwise we say that it is of complex type.  On the basis elements of \(RH\) corresponding to irreducible representations of real or quaternionic type, \(\tx\) acts as the identity.  The irreducible complex representations of complex type, on the other hand, come in pairs that are interchanged by \(\tx\).
\end{itemize}
The construction of homogeneous vector bundles above yields well-defined morphisms of semigroups:
\[
\begin{array}{ccc}
 \Rep_\F (H) & \longrightarrow & \Vect_\F (G/H) \\
  \rho & \longmapsto & E_\rho \end{array}
\]
These extend to group homomorphisms \(\alpha_\F\colon R_\F(H)\to K_\F(G/H)\), which we also denote by \(\alpha\), \(\alpha_O\) and \(\alpha_{Sp}\), respectively.  The following lemmas are well known and can be checked directly from the definitions.
\begin{lem}\label{alpha:internal-product}\label{alpha:c-r-t-q}\label{alpha:exterior-powers}
  The maps \(\alpha\), \(\alpha_O\) and \(\alpha_{Sp}\) commute with the maps \(\cx\), \(\rx\) , \(\tx\), \(\cx'\) and \(\qx\) defined on representation and K-groups.
  They also commute with tensor products and exterior powers, \ie they define \(\lambda\)-ring homomorphisms:
  \begin{align*}
    \alpha\colon &R(H)\phantom{O}\to K(G/H)\\
    \alpha_O\colon &RO(H)\to KO(G/H)\\
    \alpha_O\oplus \alpha_{Sp}\colon &  RO(H)\oplus RSp(H) \to KO(G/H)\oplus KSp(G/H)
  \end{align*}
  The third of these is of course a graded homomorphism of \(\ZZ_2\)-graded \(\lambda\)-rings.
\end{lem}

The product of two homogeneous spaces \(G_1/H_1\) and \(G_2/H_2\) is again a homogeneous space: \((G_1/H_1)\times (G_2/H_2) \cong (G_1\times G_2)/(H_1\times H_2)\).  The maps \(\alpha_\F\) also commute with the external products defined on representation and K-rings, \ie we have commutative diagrams of the form:
\[\xymatrix{
    R_{\F_1}(H_1)\otimes R_{\F_2}(H_2) \ar[r]^{\mu} \ar[d]^{\alpha_{\F_1}\otimes\alpha_{\F_2}} & R_{\F_{12}}(H_1\times H_2) \ar[d]^{\alpha_{\F_{12}}}\\
    K_{\F_1}(G_1/H_1)\otimes K_{\F_2}(G_2/H_2) \ar[r]^{\mu} & K_{\F_{12}}(G_1/H_1\times G_2/H_2)
  }\]
Here, \((\F_1,\F_2,\F_{12})\) may be any of the triples \((\CC,\CC,\CC)\), \((\RR,\RR,\RR)\), \((\RR,\HH,\HH)\), \((\HH,\RR,\HH)\) or \((\HH,\HH,\RR)\).
The commutativity of the square follows from the previous \namecref{alpha:internal-product} and the fact that \(\alpha_\F\) commutes with the pullback maps along the projections to each factor. More generally, the maps \(\alpha_\F\) are natural in the following sense:

\begin{lem}[Naturality]\label{alpha:naturality}
  Let \(G\) and \(G'\) be compact Lie groups with closed subgroups \(H\) and \(H'\).
  Let \(f\colon G'\to G\) be a homomorphism of Lie groups such that \(f(H')\subset H\), and let \(\bar f\colon G'/H'\to G/H\) be the induced map.  Then we have commutative diagrams as follows:
  \[\xymatrix{
      {R_\F(H)} \ar[r]^{\alpha^{H'}_\F} \ar[d]^{f^*} & K_\F(G/H) \ar[d]^{\bar f^*}\\
      {R_\F(H')} \ar[r]^{\alpha^H_\F} & K_\F(G'/H')
    }\]
\end{lem}

In the next two subsections, we discuss some general situations in which \(\alpha_\F\) is known to be surjective.  Of course, given a homogeneous space \(M\), the surjectivity of \(\alpha_\F\) may depend on the chosen presentation of \(M\) as \(G/H\).  For simply connected \(M\), we may restrict our attention to presentations with simply connected \(G\):

\begin{lem}\label{suffices-to-consider-simply-connected-G}
Let \(M\) be a simply connected compact homogeneous space.
If there is any presentation of \(M\) as \(G/H\) such that \(\alpha\colon R(H)\to K(M)\) is surjective, then there is also such presentation with simply connected \(G\) and connected \(H\).
\end{lem}
Indeed, this is immediate from \cref{presentations-with-simply-connected-G} and naturality (\cref{alpha:naturality}).

\newcommand{\Tor}{\mathrm{Tor}}
\subsection{Complex bundles:  Pittie and Steinberg's Theorem}\label{S: improve Pittie}
It turns out that the complex version of \(\alpha\) is surjective in many interesting cases:

\begin{thm}\label{Improvement of Pittie's}
Let \(G\) be a compact connected Lie group with \(\pi_1(G)\) torsion-free, and let \(H\subset G\) be a closed connected subgroup such that \(\rank G - \rank H \leq 1\).  Then the natural map \(\alpha\colon RH\to K^0(G/H)\) is surjective and induces an isomorphism
\[
 \ZZ \otimes_{RG} RH \cong K^0(G/H).
\]
If \(\rank H = \rank G\), then moreover \(K^0(G/H)\) is a finitely generated free abelian group, and \(K^1(G/H)=0\).
\end{thm}
\begin{rem}\label{REM: optimal Pittie's}
  The rank assumptions here are optimal:
  \begin{compactitem}[--]
  \item
     When \(\rank G-\rank H=1\),
    the higher K-group \(K^1(G/H)\) may no longer vanish, and \(K^0(G/H)\) may not be free.  For example, this happens for the Berger space \(B^{13}\) (\cref{B13:K-additive}).
  \item
    When \(\rank G-\rank H= 2\),
    even the first part of the result no longer holds. Consider for example \(G=\Spin(4)=\SS^3\times\SS^3\) and \(H\) the trivial subgroup. The map \(\alpha\) is clearly trivial, whereas \(K(\SS^3\times\SS^3)= [\Z]\oplus\Z \) is not trivial.
  \end{compactitem}
\end{rem}

When \(\rank H = \rank G\), \cref{Improvement of Pittie's} is well-known. In many cases, it was already verified by Atiyah and Hirzebruch in the paper that founded topological K-theory \cite{AH:bundles homogeneous}*{Thm~5.8}. In general, it can be obtained from a spectral sequence due to Hodgkin and a theorem of Pittie and Steinberg.  We recall both these ingredients here, derive a corollary of the latter result, and explain how to obtain \cref{Improvement of Pittie's} under the stated weaker hypothesis that \(\rank G - \rank H \leq 1\).

Hodgkin's spectral sequence is as follows:

\begin{thm}[\citelist{\cite{Hodgkin}\cite{McLeod}*{Thm~4.1}}]
Suppose \(G\) is a compact, connected Lie group with torsion-free fundamental group. Then for \(G\)-spaces \(X\) and \(Y\) we have a strongly convergent spectral sequence
\[
  E_2^{p,q} = \{\Tor_{-p}^{RG}(K_G^*(X),K_G^*(Y))\}^q
 \quad\Rightarrow\quad K^{p+q}_G(X\times Y).
\]
(The \(G\)-spaces are assumed to be compactly generated and have the homotopy type of a CW complex, see \cite{Hodgkin}*{\S\,I.1}.)
\end{thm}

\begin{figure}
  \begin{adjustwidth}{-5em}{-5em}
    \begin{center}
      \begin{tikzpicture}[x=7em,y=5ex]
        \renewcommand{\d}{|[fill=gray!15!white,minimum width=7em,minimum height=5ex]|}% diagonal
        \newcommand{\nr}[1]{|[minimum width=1ex,minimum height=1ex]|{\text{\small\texttt{#1}}}}% number
        \matrix (m)
        [
        matrix of math nodes,
        nodes in empty cells,
        % nodes={outer sep=-2pt},
        column sep=0ex,
        row sep=0ex
        ]
        {
          \cdots    & \d 0                & 0                      & \d 0                & 0                                        & \nr{1}      \\
          \d \cdots & \Tor_3^{RG}(\ZZ,RH) & \d \Tor_2^{RG}(\ZZ,RH) & \Tor_1^{RG}(\ZZ,RH) & \d \ZZ\otimes_{RG} RH                    & \nr{0}      \\
          \nr{\dots} & \nr{-3}             & \nr{-2}                & \nr{-1}             & \nr{0}                    & \quad\strut \\
        };
        % _____Axes______________
        % \node at (m-3-6) {A};
        \path[arrow]
        (m-3-6.west) edge ($(m-3-6.west)+(0,2.5)$);
        \node at ($(m-3-6.west)+(0.1,2.6)$) {\small\texttt{q}};
        \path[arrow]
        ($(m-3-1.north)+(-0.5,0.3)$) edge ($(m-3-1.north east)+(4.6,0.3)$);
        \node at ($(m-3-1.north east)+(4.7,0)$) {\small\texttt{p}};
      \end{tikzpicture}
    \end{center}
  \end{adjustwidth}
  \caption{The \(E_2\)-page of Hodgkin's spectral sequence converging to \(K^*(G/H)\).  The shaded entries compute \(K^0(G/H)\), the remaining entries \(K^1(G/H)\).}
  \label{fig:Hodgkin-SS}
\end{figure}
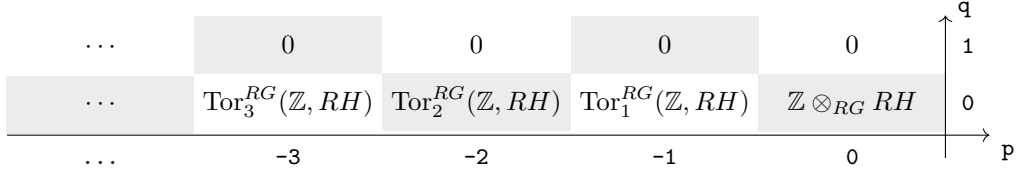

In this theorem, the complex representation ring \(RG\) is to be viewed as a \(\ZZ_2\)-graded ring concentrated in degree zero,  the equivariant K-groups \(K_G^*(X)\) and \(K_G^*(Y)\) are graded modules over this graded ring, and the Tor functors \(\Tor_{-p}^{RG}\) are to be taken in the category of such graded modules.
The \(q\)-coordinate refers to the internal grading of these Tor modules.
To prove \cref{Improvement of Pittie's}, one applies the spectral sequence to \(X := G\) and \(Y := G/H\).  Then \(K^*_G(X) \cong \ZZ\), \(K^*_G(Y) \cong RH\) and \(K^*_G(X\times Y) \cong K^*(G/H)\).
In particular, all rings and modules appearing on the \(E_2\)-page are concentrated in degree zero, so that we can replace the graded Tor functor by the usual Tor functor.
The spectral sequence thus takes the simple form displayed in \cref{fig:Hodgkin-SS}.
Moreover, the edge homomorphism \(E_2^0 = \ZZ\otimes_{RG} RH\to K^0(G/H)\) is induced by \(\alpha\) \cite{Hodgkin}*{Lemma~9.1}.
The contribution of Pittie and Steinberg is that, when \(\rank H = \rank G\), all the other groups  on the \(E_2\)-page vanish:

\begin{thm}[\cites{Pittie:Homogeneous,Steinberg:Pittie}]\label{thm:Steinberg}
  Let \(G\) be as in \cref{Improvement of Pittie's}.  For every closed connected subgroup \(H\subset G\) of maximal rank, \(RH\) is a finitely generated free \(RG\)-module.
\end{thm}

This implies \cref{Improvement of Pittie's} in the maximal rank case.

Note that the representation ring of Lie group \(G\) as above is isomorphic to the tensor product of a polynomial ring with a ring of Laurent polynomials.  Therefore, a finitely generated \(RG\)-module is free if and only if it is projective if and only if it is flat:  the first two notions coincide by a generalization of the Quillen-Suslin theorem \citelist{\cite{Swan:Laurent}\cite{Swan:Gubeladze}}, the second two coincide for finitely generated modules over arbitrary Noetherian rings.  In particular, the notions of ``projective dimension'' and ``weak\slash flat dimension'' are equivalent over such \(RG\).  Below, we always write ``projective dimension'', but we use \(\Tor\) rather than \(\Ext\) in all arguments because the \(\Tor\)-functors are the ones we are ultimately interested in.

Starting from Pittie and Steinberg's result, we prove:
\begin{cor}\label{cor:proj-dim}
  Let \(G\) be as in \cref{Improvement of Pittie's}, and let \(H\subset G\) be an arbitrary closed connected subgroup.  Then the projective dimension of \(RH\) as an \(RG\)-module is at most \(\rank G - \rank H\).
\end{cor}

The corollary implies, in particular, that for \(G\) and \(H\) as stated we have:
\begin{equation}\label{eq:Pittie:improved-Tor-vanishing}
  \Tor_p^{RG}(RH,\ZZ) = 0 \text{ for } p > \rank G -\rank H
\end{equation}
Together with Hodgkin's spectral sequence (\cref{fig:Hodgkin-SS}), this implies the general form of \cref{Improvement of Pittie's}.

The vanishing result \eqref{eq:Pittie:improved-Tor-vanishing} was observed even before Pittie announced his result by Snaith \cite{Snaith:Homogeneous}*{Thm~4.2}.  His strategy for passing from full rank to arbitrary rank is similar to the strategy used in the proof below.  However, in the absence of Steinberg and Pittie's result, Snaith had to work with completions.

\begin{proof}[Proof of \cref{cor:proj-dim}]
We use the usual notation \(\projdim_R M\) to denote the projective dimension of a module \(M\) over a ring \(R\).
Consider first the case that \(H\subset G\) is an inclusion of tori \(T' \subset T\).
Then \(T'\) is a direct factor of \(T\) \cite{HilgertNeeb}*{Lem.~15.3.2}, and \(R(T)\cong R(T')[x_1^{\pm 1},\dots,x_m^{\pm 1}]\), where \(m\) is the difference of ranks between \(T\) and \(T'\).
In this case, the projective dimension  \(\projdim_{RT}RT'\) is exactly \(m\):
we need only apply the First Change of Rings Theorem \cite{Weibel}*{Thm~4.3.3} inductively to the elements \((x_1 - 1), (x_2 - 1), \dots, (x_m - 1)\).

Next, consider the case of a torus \(T'\) contained in an arbitrary compact Lie group \(G\) with torsion-free fundamental group.
Choose a maximal torus \(T\) of \(G\) such that \(T'\subset T\), so that we may again view \(RT'\) as an \(RT\)-module.
Applying the General Change of Rings Theorem \cite{Weibel}*{Thm~4.3.1} to the restriction map \(RG\to RT\), we obtain the inequality
\(
  \projdim_{RG}RT' \leq \projdim_{RT}RT' + \projdim_{RG} RT
\).
By \cref{thm:Steinberg}, the second summand on the right side vanishes, and by the previous step \(\projdim_{RT}(RT')\leq \rank G-\rank T'\).
So altogether we obtain the required inequality.

Finally, for an arbitrary closed connected subgroup \(H\subset G\), consider the inclusion of a maximal torus \(i\colon T' \hookrightarrow H\).
By \cite{Atiyah:Bott}*{Rem.~1 after Prop.~4.9}, there exists a morphism of \(RH\)-modules \(i_*\colon RT' \to RH\) providing a left inverse to the restriction morphism \(i^*\).
In the situation at hand, we can view \(i_*\) and \(i^*\) as morphisms of \(RG\)-modules.
This shows that the vanishing of \(\Tor_p^{RG}(RT', -)\) implies the vanishing of \(\Tor_p^{RG}(RH,-)\), and hence that \(\projdim_{RG}RH \leq \projdim_{RG}RT'\).
\end{proof}

\begin{rem}
  The results on the projective dimensions can be extended to a slightly larger class of groups.
  Let us say that a Lie group is \emph{good} if it is compact and connected and if its complex representation ring \(RG\) is isomorphic to a polynomial ring tensored with a Laurent ring.  As mentioned above, all compact connected Lie groups whose fundamental group is torsion-free are good.  More precisely, a compact connected Lie group is good if and only if each simple factor of its semisimple part is either simply connected or of the form \(SO(2n+1)\) \cite{Steinberg:Pittie}.
  Steinberg's theorem in fact says:
  \begin{quote}
    Let \(G\) be a compact connected Lie group.  Then \(RH\) is free over \(RG\) for every closed connected subgroup \(H\subset G\) of maximal rank if and only if \(G\) is good.
  \end{quote}
  \Cref{cor:proj-dim} and its proof above likewise hold more generally for good~\(G\).
  However, in \cref{Improvement of Pittie's}, the assumption that \(\pi_1(G)\) is torsion-free \emph{cannot} be weakened to ``\(G\) is good''.
  For example, consider \(SO(2n+1)/SO(2n)\cong \Spin(2n+1)/\Spin(2n) \cong \SS^{2n}\).
  Both \(SO(2n+1)\) and \(\Spin(2n+1)\) are good, but the natural map \(RH \to K(\SS^{2n})\) referred to in the theorem is only surjective for the presentation with \(G=\Spin(2n+1)\).
  Indeed, by \cref{alpha:naturality} the pullback along the covering homomorphism \(\pi\colon \Spin(2n)\twoheadrightarrow SO(2n)\) fits into a commutative triangle:
  \[\xymatrix@R=0pt@C=4em{
      R(\Spin(2n)) \ar[dr]^-{\alpha(\Spin)} \\
      & K(\SS^{2n})\\
      R(SO(2n)) \ar[uu]^{\pi^*}\ar[ur]_-{\alpha(SO)}
    }\]
  The theorem implies that \(\alpha(\Spin)\) is an epimorphism that identifies \(K(\SS^{2n})\) with \(\ZZ[\reduced{\Delta}_+]/(\reduced{\Delta}_+^2)\), where \(\reduced{\Delta}_+=\Delta_+-\rank(\Delta_+)\) is one of the reduced half-spin representations.  This generator does not lie in the image of \(\alpha(SO)\).
\end{rem}

\subsection{Real bundles}\label{sec:homogeneous:real}
Suppose that \(G\) and \(H\) are as above: \(G\) is a compact connected Lie group with torsion-free fundamental group, and \(H\) is a closed connected subgroup.  The real and quaternionic versions of \(\alpha\) are surjective only under far more restrictive conditions than those of \cref{Improvement of Pittie's}.  In \cite{Seymour}*{Thm\,5.24}, Seymour describes how to compute the cokernel of \(\alpha_O\colon RO(H)\to KO(G/H)\) based on a detailed understanding of how the complex version of \(\alpha\) interacts with the involution on \(K(G/H)\) and the decomposition of \(R(H)\) into representations of different types.  The precise result is both difficult to state and to work with, but it has the following simple corollary, of which we here provide a short and independent proof:

\begin{prop}[\cite{Seymour}*{Cor.\,5.25}]\label{PROP: surjectivity alpha_O for rkG rk H conjugation}
  Let \(G\) and \(H\) be as in \cref{Improvement of Pittie's}, with \(\rank G = \rank H\). If the conjugation on \(R(H)\) is the identity map, then both
\(\alpha_O\colon RO(H)\to KO(G/H)\) and \(\alpha_{Sp}\colon RSp(H)\to KSp(G/H)\) are surjective.
\end{prop}
\begin{rem}\label{rem:t-trivial-on-R}
The condition on \(R(H)\) is satisfied for a large class of Lie groups:  conjugation acts trivially on \(R(H)\) for any compact connected Lie group \(H\) that has no circle factor and no simple factors of types \(A_n\) with \(n\geq 2\), \(D_{2n+1}\) with \(n\geq 4\) or \(E_6\). For simply connected \(H\), this follows from \cite{Z:KO-rings}*{Table~1}.  In general, \(H\) will be a quotient of a simply connected group \(\tilde H\) by a finite central subgroup, and then \(R(H)\) will embed into \(R(\tilde H)\); see the proof of \cref{B13:R(H)}.
\end{rem}

On the other hand, the condition on \(R(H)\) is by no means a necessary one.  For example, \(\alpha_O\) and \(\alpha_{Sp}\) are both surjective for even-dimensional projective spaces \(\CP^{2n}=U(2n+1)/(U(2n)\times U(1))\), even though the conjugation acts nontrivially on \(R(U(2n)\times U(1))\).

\begin{proof}
  As \(K^1(\tfrac{G}{H})\) vanishes, we can apply \cref{Bousfield's-lemma}.  Thus, the different versions of \(\alpha\) fit into the following exact diagram with exact rows
  (compare \cite{Z:KO-rings}*{Prop.\,2.2}):
  \[\xymatrix{
      {R(H)\oplus R(H)} \ar[r]\ar[d]^{\alpha\oplus\alpha}
      & RO(H)\oplus RSp(H) \ar[r]\ar[d]^{\alpha_O\oplus \alpha_{Sp}}
      & h^+(RH) \ar[r] \ar[d]^{h^+(\alpha)}
      & 0
      \\
      {K(\tfrac{G}{H})\oplus K(\tfrac{G}{H})} \ar[r]
      & KO(\tfrac{G}{H})\oplus KSp(\tfrac{G}{H}) \ar[r]
      & h^+(K(\tfrac{G}{H})) \ar[r]
      & 0
    }\]
   We already know that \(\alpha\) is surjective.
   As the involution on \(R(H)\) is trivial, so is the involution on \(K(\tfrac{G}{H})\). Therefore, \(h^+(\alpha)\) is obtained from \(\alpha\) by tensoring with \(\ZZ_2\).  In particular, \(h^+(\alpha)\) is also surjective.
   The surjectivity of \(\alpha_O\) and \(\alpha_{Sp}\) now follows from the diagram.
\end{proof}

Note that the class of spaces satisfying the hypotheses in \cref{PROP: surjectivity alpha_O for rkG rk H conjugation} is closed under products: if the conjugation acts as the identity both on \(R(H_1)\) and \(R(H_2)\), then it also acts as the identity on \(R(H_1\times H_2)\cong R(H_1)\otimes R(H_2)\). In fact, we have the slightly more general ``product theorem'':

\begin{prop}\label{surjectivity of alpha_O for products}
  Let \(G_1\), \(H_1\)  and \(G_2\), \(H_2\) be as in \cref{Improvement of Pittie's}, with \(\rank G_i = \rank H_i\).  Then \(\alpha_O\) and \(\alpha_{Sp}\) are surjective for the product \(G_1/H_1 \times G_2/H_2\) if and only if they are surjective for both \(G_1/H_1\) and \(G_2/H_2\) and one of \(h^-(K(G_1/H_1))\), \(h^-(K(G_2/H_2))\) is zero.
\end{prop}

\begin{rem}\label{REM: satisfying h-()}
The condition \(h^-(K(M))=0\) is satisfied in each of the following cases:
\begin{compactitem}[--]
\item When \(M\) is any of the homogeneous spaces included in \cref{PROP: surjectivity alpha_O for rkG rk H conjugation}, since for these conjugation acts trivially on \(K(M)\) and \(K(M)\) is torsion-free by \cref{Improvement of Pittie's}.
\item When \(M=\CP^{2n}\).  This can be verified directly, or using \cref{Bousfield's-lemma} and the known results for \(KO^{2i}(\CP^{2n})/\rx\) \cite{Z:KO-rings}*{\S\,4.2}.
\item When the cohomology of \(M\) is concentrated in degrees multiples of~\(4\).
  This can be seen using the AHSS: we find that the conjugation acts trivially on each homogeneous summand of the graded ring associated with the cellular filtration on \(K(M)\), and hence that \(h^-\) vanishes on each summand.  It then follows inductively that \(h^-\) vanishes for the whole K-ring.
\end{compactitem}
In contrast, when \(M\) is \(\SS^{8n+2}\), \(\SS^{8n+6}\) or \(\CP^{2n+1}\), we have \(h^-(K(M))=\ZZ_2\).
\end{rem}

\begin{proof}[Proof of \Cref{surjectivity of alpha_O for products}]
 Consider the diagram used in the previous proof, with \(H_1\times H_2\) in place of \(H\) and \(G_1\times G_2\) in place of \(G\).
Using the Künneth theorems for complex representation rings \cite{Adams:LieLectures}*{Thm~3.65}, for K-theory \cite{Atiyah:Kunneth} and for Tate cohomology (see \cref{Kunneth-for-Tate} and \Cref{Bousfield's-lemma-for-RG}), we obtain:
 \begin{align*}
   R(H_1\times H_2)         & \cong R(H_1)\otimes R(H_2)         &   h^+(R(H_1\times H_2))         & \cong h^+(RH_1)\otimes h^+(RH_2) \\
   K\big(\tfrac{G_1}{H_1}\times\tfrac{G_2}{H_2}\big) & \cong K\big(\tfrac{G_1}{H_1}\big) \otimes K\big(\tfrac{G_2}{H_2}\big) &   h^+(K\big(\tfrac{G_1}{H_1}\times \tfrac{G_2}{H_2}\big)) & \cong h^+(K\big(\tfrac{G_1}{H_1}\big))\otimes h^+(K\big(\tfrac{G_2}{H_2}\big)) \\[0pt]
   & & & \quad\oplus h^-(K\big(\tfrac{G_1}{H_1}\big))\otimes h^-(K\big(\tfrac{G_2}{H_2}\big))
 \end{align*}

The vertical map \(\alpha_O\oplus \alpha_{Sp}\) is surjective if and only if the vertical map \(h^+(\alpha)\) on the right is surjective, which is the case if and only if it is surjective for both \(\tfrac{G_1}{H_1}\) and \(\tfrac{G_2}{H_2}\) and if the additional summand  \(h^-(K(\tfrac{G_1}{H_1}))\otimes h^-(K(\tfrac{G_2}{H_2}))\) vanishes.
\end{proof}

When one of the factors in the product is a sphere whose dimension is a multiple of eight, the other factor may be arbitrary --- the isomorphism \eqref{eq:external-product-iso} in \cref{SS:K-theory rings} and the surjectivity of \(\alpha_O\) for \(\SS^{8n}\) imply:
\begin{prop}\label{PROP:alpha_O product S8n}
If \(M=G/H\) is a compact homogeneous space such that $\alpha_O$ is surjective, then
\[
\alpha_O\colon RO(H\times \Spin(8n))\to KO(M \times \SS^{8n})
\]
is also surjective.
\end{prop}

%%% Local Variables:
%%% mode: latex
%%% TeX-master: "main"
%%% End:

%% file: Curvature.tex
\subsection{Curvature on homogeneous bundles}
Any compact Lie group \(G\) admits a biinvariant metric, which in particular has non-negative sectional curvature.
Moreover, isometric quotients of non-negatively curved manifolds also have non-negative sectional curvature by O'Neill's theorem on Riemannian submersions.
It easily follows that compact homogeneous spaces \(G/H\) and homogeneous vector bundles \(E_\rho =G\times_H\F^m\) over compact homogeneous spaces admit metrics with non-negative sectional curvature (where \(\F^m\) is endowed with the Euclidean flat metric).

We will use the following result that relates the K and KO-theory of a homogeneous space to the SCSQ.

\begin{prop}[\cite{Gonzalez:nonnegative}]\label{PROP: surjectivity of alphaO implies SISC}
Let \(M=G/H\) be a compact homogeneous space, where \(G\) is a compact Lie group, and assume that the map \(\alpha_O\colon RO(H)\to KO(M)\) (respectively \(\alpha\colon R(H)\to K(M)\)) is surjective.
Then for every real (resp. complex) vector bundle \(E\) over \(M\), the product \(E\times\R^k\) carries a metric of non-negative sectional curvature for some integer~\(k\).
\end{prop}

We give here the idea of the proof, which consists in showing that every vector bundle \(E\) is stably isomorphic to a homogeneous vector bundle (see \cite{Gonzalez:nonnegative} for the details).
By assumption, the virtual class \(E-0\in K_\F(G/H)\) equals \(E_{\rho_1}-E_{\rho_2}\) for some representations \(\rho_1,\rho_2\) of \(H\).
It is well known that one can choose a representation \(\rho_3\) such that \(E_{\rho_2}\oplus E_{\rho_3}\) is a trivial bundle \(m\).
Thus
\[
E=E_{\rho_1}-E_{\rho_2}+E_{\rho_3}-E_{\rho_3}=E_{\rho_1\oplus \rho_3}-m
\]
in \(K_\F(G/H)\), and it follows that \(E\times\R^k \cong E_{\rho_1\oplus \rho_3\oplus (k-m)} \) for some \(k\).

\subsection{Vector bundles over tangentially homotopy equivalent spaces}\label{sec:WorkHorse}

In this section we prove \cref{SCST for tangentially homotopy equivalent manifolds}.
It is a consequence of \cref{Work horse} below, which follows from classical results by Haefliger \cite{Haefliger} and Siebenmann \cite{Siebenmann}.

\begin{workhorse}[\cite{TW:moduli}*{Theorem 10.1.6}]\label{Work horse}
Let \(E_i\to M_i^n\) be
vector bundles over compact manifolds of the same rank \(l\), for \(i=1,2\).
Suppose that \(f\colon  E_1 \to E_2\) is
a tangential homotopy equivalence, where \(l\geq 3\) and \(l>n\).
Then \(f\) is homotopic to a
diffeomorphism.
\end{workhorse}
\begin{proof}[Proof of \cref{SCST for tangentially homotopy equivalent manifolds}]
Let \(f\colon M \to N\) be a tangential homotopy equivalence and \(\pi\colon E\to N\) an arbitrary smooth real (resp.\ complex) vector bundle.
We shall prove that the manifold \(E\times\R^k\) carries a metric with non-negative sectional curvature for some~\(k\).

Consider the pullback of \(E\) along~\(f\):
\[
\xymatrix{  f^*E\ar[d]^{\pi'}\ar[r]^{h} & E\ar[d]^{\pi}\\
M\ar[r]^{f}  & N}
\]
When \(E\) is a complex vector bundle, the pullback \(f^*E\) is a also a complex vector bundle.
But let us first ignore any possible complex structure on \(E\) and concentrate on the (underlying) real bundle\slash the total space of \(E\).
The map \(f\) is a homotopy equivalence, and the same holds for the bundle projections \(\pi\) and \(\pi'\).
Hence \(h\) is also a homotopy equivalence.
By our assumption on \(f\), we have  \(f^*(TN)= TM\) in \(KO(M)\).
Using the isomorphisms
\begin{align*}
  TE    & \cong \pi^*TN \oplus \pi^*E \\
  T(f^*E) &\cong (\pi')^*(TM) \oplus (\pi')^*f^*E
\end{align*}
\cite{tomDieck:Topologie}*{Satz~IX.6.9} together with the commutativity of the above square we obtain that \(h^*(TE) = T(f^*E)\) in \(KO(f^*E)\).
Thus \(h\) is also a tangential homotopy equivalence.

If necessary, replace \(E\) by \(E \times \R^k\), where \(k\) is chosen large enough to satisfy the assumptions of \cref{Work horse}.
Possibly increasing \(k\) once more, we may by our assumption on \(M\) equip \(f^*(E \times \R^k) = f^*E \times \R^k\) with a metric of non-negative sectional curvature.
Then \cref{Work horse} implies that \(h\) defines a diffeomorphism  \( f^*E \times \R^k \to E \times \R^k\), and hence \(E \times \R^k\) can be equipped with a metric of non-negative sectional curvature.
\end{proof}

\subsection{Products with the real line}\label{sec:ProductsLine}

Here we provide a proof of the claim in \Cref{EXAM:EXOTIC}.
The starting point is the following consequence of Smale's h-cobordism theorem.

\begin{lem}[\cite{Br:Structures}]\label{L:h-cobordism}
Let \(M^n,N^n\) be compact simply connected manifolds of dimension \(n\geq 5\).
If \(M\times\RR\) is diffeomorphic to \(N\times\RR\), then \(M\) is diffeomorphic to \(N\).
\end{lem}
We obtain the following equivalence for the existence of non-negatively curved metrics.

\begin{lem}\label{L:soulsCodim1}
Let \(M^n\) be a compact simply connected manifold of dimension \(n\geq 5\).
Then \(M\times\RR\) admits a metric with non-negative sectional curvature if and only if so does \(M\).
\end{lem}

\begin{proof}
One implication is trivial: if \(M\) admits a non-negatively curved metric, take the product metric on \(M\times\RR\).

Suppose conversely that \(M\times\RR\) admits a metric with non-negative sectional curvature.
Then by the Soul Theorem, \(M\times\RR\) is diffeomorphic to the total space of a vector bundle over a soul \(S\) of \(M\times\RR\).
Since any vector bundle projection is a homotopy equivalence, it follows that \(S\) is homotopy equivalent to \(M\), and in particular of the same dimension.
Hence \(M\times\RR\) is diffeomorphic to the total space of a real line bundle over \(S\), which must be \(S\times\RR\) as \(S\) is simply connected.
Using \Cref{L:h-cobordism} we see that \(M\) is diffeomorphic to \(S\).
This completes the proof: recall that a soul \(S\) of \(M\times\RR\) inherits a metric of non-negative sectional curvature, as it is a totally geodesic submanifold.
\end{proof}

%%% Local Variables:
%%% mode: latex
%%% TeX-master: "main"
%%% End:

%% file: PositiveCases.tex
We are now ready to apply the tools developed in the previous sections to the Stable Converse Soul Question (SCSQ).  Almost all statements made in \cref{S:Intro} are direct consequences of the results obtained in \cref{P: surjectivity alpha_O dim 7,P: surjectivity alpha_O products of spheres,surjectivity-of-alphaO-for-full-flags} below:

\begin{proof}[Proof of \Cref{THM: complex SCST for homogeneous}]
By \cref{Improvement of Pittie's}, the map \(\alpha\colon R(H)\to K(G/H)\) is surjective for all homogeneous spaces \(G/H\) in \Cref{THM: complex SCST for homogeneous}.  Now apply the complex versions of \crefnosort{PROP: surjectivity of alphaO implies SISC,SCST for tangentially homotopy equivalent manifolds}.
\end{proof}

\begin{proof}[Proofs of \cref{THM: SCST for flag manifolds,THM: SISC for homogeneous of dim 7,THM: SCSQ for rkG rkH and conjugation,THM: SCST for products of 4n CROSSes,THM: SCST for products of spheres,THM: SCST for products with S8}]
For all the homogeneous spaces contained in each theorem, there is a presentation \(G/H\) for which the map \(\alpha_O\colon RO(H)\to KO(G/H)\) is surjective: the summary below provides a reference for surjectivity of \(\alpha_O\) in each case.  Therefore \cref{PROP: surjectivity of alphaO implies SISC} gives the result.
\begin{center}
  \begin{tabular}{l @{$\quad\Leftarrow\quad$} l}
    \cref{THM: SCST for flag manifolds}          & \cref{surjectivity-of-alphaO-for-full-flags}\\
    \cref{THM: SISC for homogeneous of dim 7}    & \cref{P: surjectivity alpha_O dim 7}    \\
    \cref{THM: SCSQ for rkG rkH and conjugation} & \cref{PROP: surjectivity alpha_O for rkG rk H conjugation}\\
    \cref{THM: SCST for products of 4n CROSSes} & \cref{surjectivity of alpha_O for products}, Prop. 5.1 in \cite{Gonzalez:nonnegative} and \cref{REM: satisfying h-()}\\
    \cref{THM: SCST for products of spheres}     & \cref{P: surjectivity alpha_O products of spheres} \\
    \cref{THM: SCST for products with S8} & \cref{PROP:alpha_O product S8n}
  \end{tabular}
\end{center}
(\Cref{THM: SCST for products with S8} also relies on \cref{SCST for tangentially homotopy equivalent manifolds}.)
\end{proof}

\begin{proof}[Proof of the Main Theorem]
Given \cref{SCST for tangentially homotopy equivalent manifolds} and the classification of homogeneous spaces with positive curvature \cite{WZ:Revisitng}, we have to prove the statement for the CROSSes and for the spaces \(W^6\), \(B^7\), \(W^{7}_{p,q}\), \(W^{12}\) and \(W^{24}\). For any of the CROSSes, it is shown in \cite{Gonzalez:nonnegative} that there is a presentation \(G/H\) for which \(\alpha_O\colon RO(H)\to KO(G/H)\) is surjective, so \cref{PROP: surjectivity of alphaO implies SISC} gives the result. For the spaces \(W^6\), \(B^7\) and \(W^{7}_{p,q}\) the result follows from \Cref{THM: SISC for homogeneous of dim 7}. Finally, recall that the usual presentations of \(W^{12}\) and \(W^{24}\) are \(Sp(3)/ Sp(1)^3\) and \(F_4/\Spin(8)\), respectively, and observe that conjugation acts trivially on both \(Sp(1)^3\) and \(\Spin(8)\). Thus \cref{THM: SCSQ for rkG rkH and conjugation} completes the proof.
\end{proof}

\subsection{Homogeneous spaces of dimension at most seven}\label{S: Homogeneous spaces of dimension 7}

Compact simply connected homogeneous spaces of dimension \(\leq 7\) are classified (see \cites{Gorbatsevich:compact,Klaus:Thesis}). Such a space is diffeomorphic to one of the following:
\begin{itemize}
\item \(\SS^n=\Spin(n+1)/\Spin(n)$, with $2\leq n\leq 7\)
\item \(\CP^n= U(n+1)/(U(n)\times U(1))\) with \(n=2\) or \(n=3\)
\item a product of spheres, \(\CP^2\times\SS^2\), or \(\CP^2\times\SS^3\)
\item one of the following spaces appearing in the corresponding dimensions:
\begin{itemize}[AAAA]
\item[\(n=5:\)] \(Wu=SU(3)/SO(3)\), the Wu-manifold
\item[\(n=6:\)] \(W^6=SU(3)/T^2\), the Wallach flag manifold
\item[\(n=6:\)] \(\widetilde{Gr}_{2,5} =SO(5)/(SO(2)\times SO(3))\), the oriented real Grassmannian
\item[\(n=7:\)] \(T_1\SS^4 =\Spin(5)/\Spin(3)\), the unit tangent bundle of \(\SS^4\) \\
  (Here, the inclusion \(\Spin(3)\subset \Spin(5)\) is the obvious one.)
\item[\(n=7:\)] \(B^7=\Spin(5)/\Spin(3)\), the Berger space \\
  (Here, the inclusion of \(\Spin(3)\) is defined by a nontrivial \(5\)-dimensional real representation.)
\item[\(n=7:\)] \(W_{p,q}^7= SU(3)/T_{p,q}^1\), the Aloff-Wallach spaces
\item[\(n=7:\)] \(N_{p,q}^7= (SU(2)\times SU(3))/(U(1)\times SU(2))_{p,q}\), the Witten manifolds
\item[\(n=7:\)] \(N_{p,q,r}^7=SU(2)^3/T_{p,q,r}^2\)
\item[\(n=7:\)] \(\SS^2\times Wu\)
\end{itemize}
\end{itemize}
\(T^1\) and \(T^2\) denote the circle and the torus groups respectively. The subscripts in the quotient groups denote different inclusions parametrized by coprime integers \(p,q,r\).

\label{sec: cohomology}
\begin{table}
  \renewcommand{\arraystretch}{1.3}
  \begin{adjustwidth}{-5em}{-5em}
    \[
    \begin{array}{l@{\hspace{2em}} cccccccc@{\hspace{2em}}l}
      \toprule
      \text{space}       & H^0 & H^1 & H^2        & H^3  & H^4               & H^5        & H^6 & H^7 & \text{ reference }                                               \\
      \midrule
      Wu                 & \Z  & 0   & 0          & \Z_2 & 0                 & \Z         &     &     & \text{ \cite{Barden:simply}*{pp.\,9--10}} \\%\text{\citelist{\cite{Barden:simply}\cite{MimuraToda}*{p.\,149}}}
      W^6                & \Z  & 0   & \ZZ^2 & 0    & \ZZ^2        & 0          & \Z  &     & \text{ \cite{AZ:Geometrically}*{p.\,1005} }                                  \\
      \widetilde{Gr}_{2,5} & \Z  & 0   & \Z         & 0    & \Z                & 0          & \Z  &     & \text{ \cite{Vanzura:cohomology}*{Thm~11} }                                \\
      B^7                & \Z  & 0   & 0          & 0    & \Z_{10}           & 0          & 0   & \Z  & \text{ \cite{CE:classification}*{p.\,365} }                                 \\
      T_1\SS^4    & \Z  & 0   & 0          & 0    & \Z_2              & 0          & 0   & \Z  & \text{ \cite{CE:classification}*{p.\,379} }                                 \\
      W_{p,q}^7          & \Z  & 0   & \Z         & 0    & \Z_{|p^2+q^2+pq|} & \Z         & 0   & \Z  &    \text{ \cite{E:new}*{Prop.~36}   }  \\%\text{ \cites{AW:an infinite, E:new}   }
      N_{p,q}^7          & \Z  & 0   & \Z         & 0    & \Z_{q^2}          & \Z         & 0   & \Z  & \text{ \cite{KS:adiffeo}*{p.\,375} }                                        \\
      N_{p,q,r}^7        & \Z  & 0   & \ZZ^2 & 0    & \Z_{|2pqr|}       & \ZZ^2 & 0   & \Z  & \text{ \cite{Bermbach:Thesis}*{p.\,85} }                           \\
      \bottomrule
    \end{array}
    \]
  \caption{Cohomologies of homogeneous spaces of dimension at most seven}\label{table: cohomology-of-homogeneous}
  \end{adjustwidth}
\end{table}

\Cref{table: cohomology-of-homogeneous} displays the integer cohomology groups \(H^k(-,\Z)\) of the spaces described above, together with the corresponding references.
Using \cref{P: description of manifolds leq 7}, we easily deduce:

\begin{lem}\label{L: realification dim 7}
The reduced realification map \(\rx\colon \reduced{K}(M)\to \reduced{KO}(M)\) is surjective for all compact simply connected homogeneous spaces of dimension at most \(7\), except for \(Wu\) and \(\SS^2\times Wu\).
\end{lem}

\begin{proof}
As the integer cohomology of \(\SS^n\) and \(\CP^m\) has no torsion, the same holds for any possible product \(M= \SS^{n_1}\times\dots\times\SS^{n_i} \times\CP^{m_1}\times\dots\times\CP^{m_j}\).
In particular, they all satisfy the hypothesis on \(H^{n-2}\) in \cref{P: description of manifolds leq 7}.
All the spaces in \cref{table: cohomology-of-homogeneous} except for \(Wu\) likewise satisfy the hypothesis on \(H^{n-2}\). So for all of these spaces, the reduced realification is surjective.
For \(Wu\), the same \namecref{P: description of manifolds leq 7} shows that the reduced realification is \emph{not} surjective.
For the only remaining space, \(\SS^2\times Wu\), we will see non-surjectivity in the proof of \cref{P: surjectivity alpha_O dim 7} below.
\end{proof}

The ranks of the Lie groups involved in the spaces listed above are:
\begin{align*}
\rank\left( U(n) \right) = \rank \left( SU(n) \right) +1 & = n \\
 \rank \left( \Spin(2n+1)\right)=\rank \left( \Spin(2n)\right) & = n
\end{align*}
It is immediate to check which spaces satisfy \(\rank G - \rank H \leq 1\), for which \cref{Improvement of Pittie's} gives the following result.

\begin{lem}\label{L: surjectivity alpha dim 7}
Let \(G/H\) be a compact simply connected homogeneous space of dimension at most \(7\) and nondiffeomorphic to \(\SS^3\times\SS^3\), with \(G\) compact and simply connected. Then the map \(\alpha\colon  R(H)\to K(G/H)\) is surjective.%
\end{lem}
Moreover, we show that the real version \(\alpha_O\) is surjective for all homogeneous spaces of dimension at most \(7\).

\begin{prop}\label{P: surjectivity alpha_O dim 7}
Let \(G/H\) be a compact simply connected homogeneous space of dimension at most \(7\) with \(G\) compact and simply connected. Then the map \(\alpha_O\colon  RO(H)\to KO(G/H)\) is surjective.%
\end{prop}

\begin{rem}\label{optimility of P: surjectivity alpha_O dim 7}
  \Cref{P: surjectivity alpha_O dim 7} is optimal in the sense that in dimension \(n=8\) it is no longer true: by \cref{surjectivity-of-alphaO-for-full-flags} below, surjectivity of \(\alpha_O\) fails for the presentation \(SU(2)^{\times 4}/T\) of the product \(\SS^2\times\SS^2\times\SS^2\times\SS^2\).
\end{rem}

\begin{proof}[Proof of \cref{P: surjectivity alpha_O dim 7}]
When \(G/H\) is nondiffeomorphic to \(\SS^3\times\SS^3,Wu\) or \(\SS^2\times Wu \), the result follows immediately from \cref{L: realification dim 7,L: surjectivity alpha dim 7}.

For \(G/H=\SS^3\times\SS^3\), recall that \( KO(\SS^3\times\SS^3 )=[\Z]\) (see \cref{lem: triviality ok KO product spheres}), hence \(\alpha_O\) is surjective for trivial reasons.

It remains to consider the cases when \(G/H\) is diffeomorphic to \(Wu\) or \(\SS^2\times Wu\).

For \(G/H=Wu\), the AHSS's and the description of the cohomology in \cref{table: cohomology-of-homogeneous} show that \(\reduced K(Wu)\) vanishes while \(\reduced{KO}(Wu) \cong \Z_2\), so clearly realification is not surjective in this case. An explicit generator of \(\reduced{KO}(Wu)\) is given by the reduced stable class of the tangent bundle of \(Wu\).  Indeed, by (\ref{AHSS-facts:characteristic-classes}) in \cref{sec:AHSS}, the second Stiefel-Whitney class defines an isomorphism
\(
\reduced{KO}(Wu) \to H^2(Wu; \Z_2)
\),
and it is known that the second Stiefel-Whitney class of \(Wu\) is nontrivial (see for example \cite{Barden:simply}*{Lem.~1.1}, where the Wu-manifold is denoted by \(X_{-1}\)). On the other hand, the tangent bundle of a homogeneous space is a homogeneous vector bundle (see \cref{SS: homogeneous vector bundles}). This implies the surjectivity of \(\alpha_O\colon  RO(H)\to KO(Wu)\).

\medskip

Essentially the same argument works for the product \(M=\SS^2\times Wu=G/H\). Its cohomology is easily computed using the Künneth theorem.
The AHSS's identify \(\reduced K(M)\) and \(\reduced KO(M)\) with the kernels of the corresponding differentials \(d_3\colon E_2^{2,-2}\to E_2^{5,-4}\), both of which vanish according to the Leibniz rule.  Using above the description of the realification map above \cref{lem: realification argument}, we can thus identify the cokernel of \(\rx\colon \reduced K(M)\to \reduced{KO}(M)\) with the cokernel of the reduction map \(\mathsf{red}\colon H^2(M,\ZZ)\to H^2(M,\ZZ_2)\).  This cokernel in turn can be identified with the \(\Z_2\)-term of \(H^2(M, \Z_2)\) corresponding to \(H^2(Wu,\Z_2 )\) under the Künneth formula.
Since the second Stiefel-Whitney class of \(\SS^2\) vanishes, it follows from the Whitney product formula that the cokernel of \(\mathsf{red}\) is generated by the second Stiefel-Whitney class of \(M\). So again, the reduced tangent bundle generates the cokernel of \(\rx\), and hence  \(\alpha_O\colon RO(H)\to KO(M)\) is surjective.
\end{proof}

\subsection{Some products of spheres}\label{SS: Products of cohomology spheres}
A product of spheres is a homogeneous space \(G/H\) and the difference \(\rank G-\rank H\) equals the number of odd-dimensional factors.

\begin{prop}\label{P: surjectivity alpha_O products of spheres}
Each of the products of spheres in \cref{THM: SCST for products of spheres} has a presentation  \(G/H\) for which the map \(\alpha_O\colon  RO(H)\to KO(G/H)\) is surjective.%
\end{prop}

\begin{proof}
\newcommand{\evennr}{\text{\textnormal{\small even}}}
\newcommand{\oddnr}{\text{\textnormal{\small odd}}}
The case \(\ell=1\) has already been discussed in \cite{Gonzalez:nonnegative}.
Here, the most delicate case is the case of spheres of dimensions \(n\equiv 1\):  it is a result of Rigas \cite{R:geodesics} that in this case \(\alpha_O\) is surjective for some presentation of \(\SS^n\), though it is not apparent which presentation this is.  In all other dimensions, and for all products of spheres considered below, the usual presentation of (products) of \(\SS^n\) as (products of) \(\Spin(n+1)/\Spin(n)\) does the job, and we implicitly use this presentation without further specification.

For the tuples \(\{3,3\},\{7,7\},\{7,7,7\}\), the product \(G/H\) satisfies \(KO(G/H)=[\Z]\) by \cref{lem: triviality ok KO product spheres}, hence \(\alpha_O\colon RO(H)\to KO(G/H)\) is trivially surjective. For the remaining tuples in the cases \(\ell=2,3,4\), we have \(\rank G-\rank H\leq 1\), so that \(\alpha\colon R(H)\to K(G/H)\) is surjective by \cref{Improvement of Pittie's}. All of these remaining tuples (except for those of the form \(\{0,n\}\) and \(\{0,\evennr,\oddnr\}\), which are covered by the last item) also satisfy the numerical condition of \cref{lem: surjectivity of realif product spheres}, hence realification \(\rx\colon  K(G/H)\to KO(G/H)\) is surjective. Thus we get the surjectivity of \(\alpha_O\colon RO(H)\to KO(G/H)\).

Now let us consider the tuples \(\{\evennr,n_2,\dots,n_\ell\}\) where \(n_i\equiv 0\textrm{ or }4\) for all~\(i\).
For the product \(G_2/H_2\) corresponding to the tuple \(\{n_2,\dots, n_{\ell}\}\), \Cref{PROP: surjectivity alpha_O for rkG rk H conjugation} implies that the map \(\alpha_0\colon RO(H_2)\to KO(G_2/H_2)\) is surjective and that \(h^-(K(G_2/H_2))=0\) (see \Cref{REM: satisfying h-()}). Observe that for the usual presentation of an even dimensional sphere \(G_1/H_1\) we have that \(\rank G_1=\rank H_1\) and  \(\alpha_O\colon RO(H_1)\to KO(G_1/H_1)\) is surjective, hence \cref{surjectivity of alpha_O for products} applies to the product \(G/H=G_1/H_1\times G_2/H_2\) and shows that \(\alpha_O\colon RO(H)\to KO(G/H)\) is surjective.

Finally, consider \(A\cup \{n_k,\dots, n_{\ell}\}\), where \(n_i\equiv 0\) for all \(i\) and \(A\) is any of the tuples above.  We have already shown that \(\alpha_O\) is surjective for the appropriate presentation of the product corresponding to the tuple \(A\), so the claim follows directly from \cref{PROP:alpha_O product S8n}.
\end{proof}

\subsection{Few full flag manifolds}\label{SS: Flag manifolds}
When \(T\subset G\) is a maximal torus, the homogeneous space \(G/T\) is referred to as a full flag manifold.  As always, we may and will assume that \(G\) is simply connected. The KO-theory such full flag manifolds is investigated in \cite{Z:KO-rings}.  The results there easily imply:

\begin{prop}\label{surjectivity-of-alphaO-for-full-flags}
  The only full flag varieties for which \(\alpha_O\colon R(T)\to KO(G/T)\) is surjective are
  \(SU(m)/T\) with \(m\in\{ 2,3,4,5,7\}\), \(\Spin(7)/T\) and \(G_2/T\), and the following products:
  \begin{align*}
    SU(2)^{\times 2}  & /T & SU(3)^{\times 2} & /T & (SU(3)& \times SU(4))/T & (SU(3)\times G_2)& /T \\
    SU(2)^{\times 3} & /T & SU(3)^{\times 3} & /T & (SU(3)  & \times SU(5))/T
  \end{align*}
  %% correction in case of SU:
  %% \alpha_O IS ALSO surjective for SU(2)
  %%
  %% correction for \Spin:
  %%   \Spin(3) = SU(2)
  %%   \Spin(4) = SU(2) x SU(2)
  %%   \Spin(6) = SU(4)
\end{prop}
\begin{proof}
  Let \(b_\CC\), \(b_\RR\) and \(b_\HH\) denote the number of basic representations of \(G\) of complex, real and quaternionic type, respectively. By the main result of  \cite{Z:KO-rings}, the \(\ZZ_4\)-graded ``Witt ring'' \(\bigoplus_i KO^{2i}(G/T)/\rx\) is an exterior algebra over \(\ZZ_2\) on \(b_\HH\) generators of degree~\(1\) and \((b_\CC/2)+b_\RR\) generators of degree~\(3\). By \cite{Z:KO-rings}*{Example~2.3}, \(\alpha_O\) is surjective if and only if \(KO^0(G/T)/\rx =\ZZ_2\), i.\,e.\ if the only additive generator of the Witt ring in degree zero is the multiplicative unit.  This happens if and only if all of the following conditions are met:
  \begin{compactitem}[-]
  \item There are at most \(3\) generators in degree~\(1\), \ie \(b_\HH\leq 3\).
  \item There are at most \(3\) generators in degree~\(3\), \ie \((b_\CC/2) + b_\RR\leq 3\).
  \item We do not simultaneously have generators in degree~\(1\) and in degree~\(3\), \ie \(b_\HH =0 \) or \(b_\CC+ b_\RR = 0\).
  \end{compactitem}
  The claim now follows from \cite{Z:KO-rings}*{Table~1}.
  %%
  %% Possible cases with only one simple factor:
  %%
  %% | type | group   | b_C/2 + b_R | b_H |
  %% |------+---------+-------------+-----|
  %% | A_1  | SU(2)   |           0 |   1 |
  %% | A_2  | SU(3)   |           1 |   0 |
  %% | A_3  | SU(4)   |           2 |   0 |
  %% | A_4  | SU(5)   |           2 |   0 |
  %% | A_6  | SU(7)   |           3 |   0 |
  %% | B_3  | \Spin(7)|           3 |   0 |
  %% | G_2  | G_2     |           2 |   0 |
  %%
  %%
  %% So the possible products are
  %% A_1 x A_1,  A_1 x A_1 x A_1
  %% A_2 x A_2,  A_2 x A_2 x A_2
  %% A_2 x A_3
  %% A_2 x A_4
  %% A_2 x G_2
  %%
\end{proof}

%%% Local Variables:
%%% mode: latex
%%% TeX-master: "main"
%%% End:

%% file: Berger13.tex
\renewcommand{\tilde}{\widetilde}

In this section we prove \cref{THM:B13} --- see \cref{B13:KO-ring,B13:Pontryagin,B13:image-of-alphaO} below.

The space under consideration can be constructed as follows. Consider the standard embedding $j\colon Sp(2)\to SU(4)$, and define:
\begin{equation}\label{eq:B13i}
  \begin{aligned}
    i\colon Sp(2)\times S^1 & \longrightarrow \quad  SU(5) \\
    (A,z) \quad & \mapsto   \diag(j(A)z, \overline{z}^4)
  \end{aligned}
\end{equation}
The homomorphism $i$ has nontrivial kernel given by
\(
\{ (Id,1), ( -Id,-1)  \} \cong \ZZ_2
\).  So \(i\) defines a two-fold covering from \(Sp(2)\times S^1\) onto its image, which is usually denoted \(Sp(2)\times_{\ZZ_2} S^1\) or simply \(Sp(2)\cdot S^1\).  The \(13\)-dimensional \textbf{Berger space $B^{13}$} is the quotient
\[
  B^{13} := \frac{SU(5)}{Sp(2)\times_{\ZZ_2} S^1}
\]
It is the total space of a differentiable \(\RP^5\)-bundle over \(\CP^4\) \cite{Z:examples}:
\begin{equation}\label{eq:B13:fibre-bundle}
  \RP^5 \to B^{13} \to \CP^4
\end{equation}

\subsection{Representation rings}
We collect here some information on the representation rings of the groups used to construct \(B^{13}\). Let \(v\) denote the standard five-dimensional complex representation of \(SU(5)\), and let \(x\) be the standard one-dimensional representation of the circle group \(S^1\).  Let \(u\) denote the restriction of the standard four-dimensional representation of \(SU(4)\) to \(Sp(2)\), viewed as a quaternionic representation.  The complex representation rings of \(SU(5)\) and of \(Sp(2)\times S^1\) can be written as follows:
\begin{alignat*}{7}
  & R(SU(5))           &  & \cong \ZZ [v,\lambda^2 v, \lambda^3 v,\lambda^4 v] \\
  & R(Sp(2)\times S^1) &  & \cong R(Sp(2))\otimes R(S^1)                       \\
  &                    &  & \cong \ZZ[\cx' u,\lambda^2(\cx' u), x^{\pm 1}]
\end{alignat*}
There are no hidden relations:  the first is a polynomial ring, while the second is a tensor product of a polynomial ring and the ring of Laurent polynomials \(\ZZ[x^{\pm 1}]\).

    %     R(SU(5)) & = \ZZ [\Lambda_1,\Lambda_2,\Lambda_3,\Lambda_4] \\
    %     R(SU(4)) & = \ZZ [\Lambda_1,\Lambda_2,\Lambda_3]           \\
    %     R(Sp(2)) & = \ZZ [\Lambda_1,\Lambda_2]                     \\
    %     R(S^1)   & = \ZZ [x^{\pm 1}]
                     %   \end{align*}
\begin{lem}\label{B13:R(H)}
  The representation ring \(R(Sp(2)\times_{\ZZ_2}S^1)\) is the subring of \(R(Sp(2)\times S^1)\) generated by \((\cx' u)^2\), \(\lambda^2(\cx' u)\), \((\cx' u)x^{\pm 1}\) and \(x^{\pm 2}\).  We may sometimes write this ring as
  \[
    R(Sp(2)\times_{\ZZ_2} S^1) \cong \ZZ[(\cx' u)^2,\lambda^2(\cx' u), (\cx' u)x^{\pm 1}, x^{\pm 2}],
  \]
  but note that there are relations among the generators.
\end{lem}
\begin{proof}
  In general, given a finite central subgroup of a compact Lie group, \(\Gamma\subset G\), the representation ring of the quotient \(G/\Gamma\) can be computed as the ring of fixed points under the action of \(\Gamma\) on \(R(G)\).  More precisely, we have an action of \(\Gamma\) on \(R(G)\otimes_\ZZ \CC\), as follows.  Take some \(\gamma \in \Gamma\) and an irreducible
  % \COMMENT[DG]{Definition of irreducible representation somewhere? Probably not necessary, I don't know...}
  % \COMMENT[MZ]{I'd say: unnecessary.}
  representation \(\rho\).  As \(\gamma\) is central, Schur's Lemma
  % \COMMENT[DG]{Reference?}
  % \COMMENT[MZ]{Googling ``Schur's Lemma, Lie groups'' gives good results, so I'd say no reference is needed.}
  implies that \(\rho(\gamma)\) acts as multiplication by some complex scalar.  Identify \(\rho(\gamma)\) with this scalar. Define \(\gamma.\rho\) to be the element \(\rho \otimes \rho(\gamma)  \in R(G)\otimes_\ZZ \CC\), for any irreducible \(\rho\), and extend this action \(\CC\)-linearly.  A representation of \(G\) descends to a representation of \(G/\Gamma\) if and only if is fixed by this action, so:
  \[
    R(G/\Gamma) \cong (R(G)\otimes \CC)^{\Gamma}\cap R(G)
  \]
  \cite{Steinberg:Pittie}*{\S\,3}. In our case, where \(\Gamma = \{(I,1),(-I,-1)\}\) and \(G=Sp(2)\times S^1\), the action of \(\Gamma\) on the ring generators of \(R(G)\) is easily computed:  \((-I,-1)\) acts as multiplication by \(-1\) on \(\cx' u\) and \(x^{\pm 1}\), but as the identity on \(\lambda^2(\cx' u)\).  It follows that the monomials in these generators that are fixed by the action of \(\Gamma\) are those monomials \((\cx' u)^i(x)^j(\lambda^2 \cx' u)^k\) with \(i+j\) even. All these monomials are linearly independent, so the above result follows.
\end{proof}

\begin{lem}\label{B13:restricting-RG-to-RH}
  Let \(\iota \colon Sp(2)\times_{\ZZ_2} S^1 \hookrightarrow SU(5)\) denote the inclusion of the image of \(i\) \eqref{eq:B13i}.  The restriction \(\iota^*\colon R(SU(5))\to R(Sp(2)\times_{\ZZ_2}S^1)\)   is determined by:
  \begin{align*}
    \iota^* (v) & =  (\cx' u) x^{1} +  x^{-4}      &
                                                  \iota^* (\lambda^2 v) & =  \lambda^2(\cx' u) x^{2} +  (\cx' u) x^{-3}\\
    \iota^* (\lambda^3 v) & = \lambda^2(\cx' u) x^{-2} +  (\cx' u) x^{3} &
                                                                     \iota^* (\lambda^4 v) & = (\cx' u) x^{-1} +  x^{4}
  \end{align*}
\end{lem}
\begin{proof}
  The restriction of \(v\) can be read off the definition of \(i\). For the other ring generators, we observe that restriction commutes with exterior powers and use the usual formulas for manipulating exterior powers (\eg \cite{FultonLang}*{p.\,5}).
  Note that \(\lambda^3(u) = u\):  as a representation of \(SU(4)\), the standard representation \(u\) is only conjugate to \(\lambda^3 u\), but its restriction to \(Sp(2)\) is of quaternionic type, hence self-conjugate.
\end{proof}

\begin{lem}\label{B13:RO(H)-generators}
  \(RO(Sp(2)\times_{\ZZ_2} S^1)\) is generated as a ring by \(1\), \(u^2\), \(\lambda^2 u\) and the image of the realification map.
\end{lem}
\begin{proof}
  We first consider the real representation ring of the product \(\tilde H := Sp(2)\times S^1\).
  As \(u\) is a quaternionic representation, both \(u^2\) and \(\lambda^2 u\) are real representations.
  We claim that \(RO(\tilde H)\) is generated as a ring by \(1\), \(u^2\), \(\lambda^2u\) and the image of the realification map.
  To see this, recall from \cref{Bousfield's-lemma-for-RG} that
  \(\cx \) and \(\cx' \) define an isomorphism \(RO(\tilde H)/\rx  \oplus RSp(\tilde H)/\qx \cong h^+(R\tilde H)\) while \(h^-(R\tilde H)=0\).  Using \cref{Kunneth-for-Tate}, we find that:
  \begin{align*}
    h^+(R(\tilde H))
    &= h^+(R(Sp(2)))\otimes h^+(R(S^1)) \\
    &= \ZZ_2[\cx' u,\lambda^2(\cx' u)] \otimes \ZZ_2\\
    &= \ZZ_2[\cx' u,\lambda^2(\cx' u)]
  \end{align*}
  As we know which monomials in \(\cx' u\) and \(\lambda^2(\cx' u)\) are real and quaternionic, respectively, we find that \(RO(\tilde H)/\rx  = \ZZ_2[u^2,\lambda^2u]\).  This implies the claim.

  Now consider the quotient \(H := Sp(2)\times_{\ZZ_2}S^1\).  We will show that the inclusion \(RH\hookrightarrow R\tilde H\) induces an isomorphism \(h^*(RH)\cong h^*(R\tilde H)\).
  To simplify the notation, we introduce the polynomial ring \(A:= \ZZ[\lambda^2(\cx' u)]\), which we view as a subring of \(R(H)\).  \Cref{B13:R(H)} implies that, additively, we can decompose \(R(H)\) and \(R(\tilde H)\) as follows:
  \begin{alignat*}{7}
    R(\tilde H) & \cong \textstyle\bigoplus_{(i,j)\in \tilde J} A\cdot (\cx' u)^ix^j          \\
    R(H)        & \cong \textstyle\bigoplus_{(i,j)\in J} A\cdot (\cx' u)^ix^j
  \end{alignat*}
  The index set \(\tilde J\) ranges over all pairs \((i,j)\) with \(i\in\NN_0\) and \(j\in \ZZ\), while the subset \(J\subset \tilde J\) ranges only over those pairs with \(i+j\) even.
  The involution \(\tx\) is trivial on \(A\) and on \(\cx' u\), while \(\tx(x) = x^{-1}\).
  Thus:
  \begin{align*}
    h^+(RH)
                & \cong h^+\big(\textstyle\bigoplus_{i+j\text{ even}} A\cdot(\cx' u)^ix^j\big) \\
                & =h^+\big(\textstyle\bigoplus_{i\text{ even}\phantom{+j}} A\cdot(\cx' u)^i\big) \oplus
      \smash{\textstyle\bigoplus_{\substack{j>0                                                   \\i+j\text{ even}}} \underbrace{h^+\left( A\cdot(\cx' u)^ix^j \oplus \tx(A\cdot(\cx' u)^ix^j)\right)}_{0}} \\
                & =h^+(\ZZ[\lambda^2(\cx' u),(\cx' u)^2])                               \\
                & \cong h^+(R\tilde H)
  \end{align*}
  The claim of the \namecref{B13:RO(H)-generators} concerning \(RO(H)\) now follows as it did above for \(RO(\tilde H)\).
\end{proof}

\subsection{Cohomology}
The cohomology of the Bazaikin spaces is known.  For the Berger space \(B^{13}\), we have
\(
H^*(B^q,\ZZ) \cong \ZZ[\beta,\gamma]/(5\beta^3,\beta^5, \gamma^2,\gamma\beta^3)
\)
with \(\deg{\beta} = 2\) and \(\deg{\gamma} = 9\) \cite{FloritZiller:Bazaikin}*{Prop.~2.1}.
Thus, the Atiyah-Hirzebruch spectral sequences (AHSS) computing K- and KO-theory will have the following lines:
\begin{@empty}
  \newcommand{\nr}[1]{{\text{\tiny(#1)}}}
  \begin{alignat*}{20}
    &
    &  & \nr{0}
    &  & \nr{2}
    &  & \nr{4}
    &  & \nr{6}
    &  & \nr{8}
    &  & \nr{9}
    &  & \nr{11}
    &  & \nr{13}
    \\
    H^*(B^{13},\ZZ):\quad   &
    &  & \ZZ  \cdot 1             \quad
    &  & \ZZ  \cdot \beta         \quad
    &  & \ZZ  \cdot \beta^2       \quad
    &  & \ZZ_5\cdot \beta^3       \quad
    &  & \ZZ_5\cdot \beta^4       \quad
    &  & \ZZ  \cdot \gamma        \quad
    &  & \ZZ  \cdot \gamma\beta   \quad
    &  & \ZZ  \cdot \gamma\beta^2
    \\
    H^*(B^{13},\ZZ_2):\quad &
    &  & \ZZ_2\cdot 1
    &  & \ZZ_2\cdot \beta
    &  & \ZZ_2\cdot \beta^2
    &  &
    &  &
    &  & \ZZ_2\cdot \gamma
    &  & \ZZ_2\cdot \gamma\beta
    &  & \ZZ_2\cdot \gamma\beta^2
  \end{alignat*}
\end{@empty}
\begin{lem}\label{B13:Sq2}\label{B13:w-and-Wu}
  The total Stiefel-Whitney class of \(B^{13}\) is \(w(B^{13})   = 1 + \beta + \beta^2\) and its total Wu class is \(\nu(B^{13})  = 1 + \beta\).
  The second Steenrod square \(Sq^2\) operates on \(H^*(B^{13},\ZZ_2)\) as follows:
  \[
    Sq^2(\beta)   = \beta^2        \quad\quad
    Sq^2(\beta^2)   = 0        \quad\quad
    Sq^2(\gamma)   = 0        \quad\quad
    Sq^2(\gamma\beta) = \gamma\beta^2      \quad\quad
  \]
\end{lem}
\begin{proof}
  We first compute the low degree Stiefel-Whitney classes using the fibre bundle \eqref{eq:B13:fibre-bundle}.  In general, given a differentiable fibre bundle
  \(
  F\xrightarrow{i} E \xrightarrow{\pi} X,
  \)
  we have \(i^*w(E) = w(F)\), \cf \cite{BorelHirzebruch:III}*{p.\,5.1}.
  In our situation, \(F=\RP^5\) with \(H^*(\RP^5,\ZZ_2) = \ZZ_2[\alpha]/\alpha^6\) and
  \(
  w(\RP^5) = 1 + \alpha^2 + \alpha^4
  \)
  \cite{MilnorStasheff}*{Cor.~11.15 (p.\,133)}. So we find
  \begin{align*}
    w(B^{13})   & = 1 + \beta + \beta^2 + \text{ terms of degrees}\geq 9.
                  \intertext{%
                  The Wu classes \(\nu_k\) are determined by the Steenrod squares by the formula
                  \(
                  Sq^k(x) = \nu_k\cup x
                  \)
                  for \(x \in H^{13-k}(B^{13},\ZZ_2)\) \cite{MilnorStasheff}*{top of p.\,132}. As \(Sq^k\) vanishes on elements of degree less than \(k\), we necessarily have \(\nu_k = 0\) for \(k>13-k\), \ie for \(k\geq 7\).  So the total Wu class reduces to
                  }
                  \nu(B^{13}) & = 1 + \nu_2 + \nu_4.
  \end{align*}
  Both \(\nu_2\) and \(\nu_4\) and the higher Stiefel-Whitney classes are now determined by the relation between the Wu classes and the Stiefel-Whitney classes \cite{MilnorStasheff}*{Thm~11.14}:
  \[
    w_k = \nu_k + Sq^1(\nu_{k-1}) + Sq^2(\nu_{k-2}) + \cdots
  \]
  Keeping in mind that \(Sq^i\) vanishes on classes of degree less than~\(i\), we find that \(\nu_2 = \beta\), \(\nu_4 = 0\) and \(w_9=w_{11}=w_{13}=0\).

  Finally, we compute all Steenrod squares. For \(\beta\) and \(\beta^2\) the claims are clear.
  For \(\gamma\beta\), we use the formula \(Sq^2(\gamma\beta) = \nu_2\cup \gamma\beta\) defining \(\nu_2\) and the result \(\nu_2 = \beta\) from above.
  The second Steenrod square of \(\gamma\) is now determined by the Cartan formula \(Sq^2(ab) = Sq^2(a)b + Sq^1(a)Sq^1(b) + aSq^2(b)\).
\end{proof}

\subsection{K-theory: additive structure}
\begin{prop}\label{B13:K-additive}
  The complex K-groups of the Berger space \(B^{13}\) are as follows:
  \begin{align*}
    K^0(B^{13})                            & = [\ZZ]\oplus\ZZ\oplus\ZZ\oplus\ZZ_5\oplus\ZZ_5                                                 \\
    K^1(B^{13})                            & =  \ZZ\oplus\ZZ\oplus\ZZ
  \end{align*}
  The (reduced) real K-groups \(\reduced{KO}^*(B^{13})\) and the maps \(\eta\), \(\cx \) and \(\rx \) are as displayed in \cref{fig:B13:Bott}.
  \begin{figure}
    \begin{adjustwidth}{-5em}{-5em}
    \[\xymatrix@R=18pt@C=4pt{
        {\reduced{K}^0} \ar@{=}[r] \ar[d]_{\rx }        & *++{\ZZ^2\oplus\ZZ_5^2} \ar@{->>}[d] \\
        {\reduced{KO}^{-6}} \ar@{=}[r] \ar[d]_{\eta} & *++{\ZZ\oplus\ZZ_5}                  \\
        {\reduced{KO}^{-7}} \ar@{=}[r] \ar[d]_{\cx }    & *++{\ZZ^2} \ar@^{^{(}->}[d]          \\
        {\reduced{K}^{-1}} \ar@{=}[r] \ar[d]_{\rx }     & *++{\ZZ^3} \ar@{->>}[d]              \\
        {\reduced{KO}^{-5}} \ar@{=}[r] \ar[d]_{\eta} & *++{\ZZ}                             \\
        {\reduced{KO}^{-6}} \ar@{=}[r] \ar[d]_{\cx }    & *++{\ZZ\oplus\ZZ_5} \ar@{^{(}->}[d]  \\
        {\reduced{K}^0} \ar@{=}[r]                   & *++{\ZZ^2\oplus\ZZ_5^2}
      }\quad
      \xymatrix@R=18pt@C=4pt{
        {\reduced{K}^0} \ar@{=}[r] \ar[d]_{\rx }        & *++{\ZZ^2\oplus\ZZ_5^2}\ar@{->>}[d] \\
        {\reduced{KO}^{-4}} \ar@{=}[r] \ar[d]_{\eta} & *++{\ZZ\oplus\ZZ_5}                  \\
        {\reduced{KO}^{-5}} \ar@{=}[r] \ar[d]_{\cx }    & *++{\ZZ}   \ar@{^{(}->}[d]           \\
        {\reduced{K}^{-1}} \ar@{=}[r] \ar[d]_{\rx }     & *++{\ZZ^3} \ar@{->>}[d]              \\
        {\reduced{KO}^{-3}} \ar@{=}[r] \ar[d]_{\eta} & *++{\ZZ^2}                           \\
        {\reduced{KO}^{-4}} \ar@{=}[r] \ar[d]_{\cx }    & *++{\ZZ\oplus\ZZ_5} \ar@{^{(}->}[d]  \\
        {\reduced{K}^0} \ar@{=}[r]                   & *++{\ZZ^2\oplus\ZZ_5^2}
      }\quad
      \xymatrix@R=18pt@C=4pt{
        {\reduced{K}^0} \ar@{=}[r] \ar[d]_{\rx }        & *++{\ZZ^2\oplus\ZZ_5^2} \ar@{->>}[d] \\
        {\reduced{KO}^{-2}} \ar@{=}[r] \ar[d]_{\eta} & *++{\ZZ\oplus\ZZ_5}                  \\
        {\reduced{KO}^{-3}} \ar@{=}[r] \ar[d]_{\cx }    & *++{\ZZ^2} \ar@{^{(}->}[d]           \\
        {\reduced{K}^{-1}} \ar@{=}[r] \ar[d]_{\rx }     & *++{\ZZ^3} \ar@{->>}[d]              \\
        {\reduced{KO}^{-1}} \ar@{=}[r] \ar[d]_{\eta} & *++{\ZZ\oplus\ZZ_2}                  \\
        {\reduced{KO}^{-2}} \ar@{=}[r] \ar[d]_{\cx }    & *++{\ZZ\oplus\ZZ_5} \ar@^{^{(}->}[d] \\
        {\reduced{K}^0} \ar@{=}[r]                   & *++{\ZZ^2\oplus\ZZ_5^2}
      }
      \xymatrix@R=18pt@C=4pt{
        {\reduced{K}^0} \ar@{=}[r] \ar[d]_{\rx }        & *++{\ZZ^2\oplus\ZZ_5^2} \ar[d] \\
        {\reduced{KO}^0} \ar@{=}[r] \ar[d]_{\eta}    & *++{\ZZ\oplus\ZZ_5\oplus\ZZ_2}    \ar[d]^{\substack{\text{image $= \ZZ_2$} \\\text{iso on $2$-torsion}}}        \\
        {\reduced{KO}^{-1}} \ar@{=}[r] \ar[d]_{\cx }    & *++{\ZZ\oplus\ZZ_2} \ar[d]^{\substack{\text{$0$ on $\ZZ_2$}                \\\text{mono on $\ZZ$}}}                               \\
        {\reduced{K}^{-1}} \ar@{=}[r] \ar[d]_{\rx }     & *++{\ZZ^3} \ar[d]                                                          \\
        {\reduced{KO}^{-7}} \ar@{=}[r] \ar[d]_{\eta} & *++{\ZZ^2} \ar[d]^{\text{image $=\ZZ_2$}}                                  \\
        {\reduced{KO}^0} \ar@{=}[r]\ar[d]_{\cx }        & *++{\ZZ\oplus\ZZ_5\oplus\ZZ_2}  \ar[d]^{\substack{\text{$0$ on $\ZZ_2$}    \\\text{mono on $\ZZ\oplus\ZZ_5$}}}    \\
        {\reduced{K}^0} \ar@{=}[r]                   & *++{\ZZ^2\oplus\ZZ_5^2}
      }
    \]
    \caption{The Bott sequence for \(B^{13}\). Missing arrows indicate the zero map.}\label{fig:B13:Bott}
  \end{adjustwidth}
  \end{figure}
  In particular, the quaternionification \(\qx \colon \reduced{K}^0(B^{13})\to \reduced{KO}^{-4}(B^{13})\) is surjective, while the realification \(\rx \colon \reduced{K}^0(B^{13})\to \reduced{KO}^0(B^{13})\) has cokernel isomorphic to \(\ZZ_2\).
\end{prop}

\begin{proof}
  The AHSS for complex K-theory collapses: the torsion elements in \(\ZZ_5\) cannot be reached by any nontrivial differential.  This immediately determines \(K^1(B^{13})\), and it determines \(K^0(B^{13})\) up to an extension problem:  there is a subgroup \(F_5\subset K^0(B^{13})\) that fits into a short exact sequence $0\to\ZZ_5\to F_5\to\ZZ_5\to 0$.  In order to determine whether \(F_5\) is $\ZZ_5\oplus\ZZ_5$ or $\ZZ_{25}$, we consider the involution \(\tx\) on \(K^0\) and the induced morphism of spectral sequences, as described in \cref{sec:AHSS}.  We obtain, in particular, the following commutative diagram, which implies \(F_5 = \ZZ_5\oplus \ZZ_5\).
  \[\begin{tikzcd}
      0 \arrow{r}                        & \ZZ_5 \arrow{r}\arrow{d}{\id} & F_5 \arrow{r}\arrow{d}{\tx} & \ZZ_5 \arrow{r} \arrow{d}{-\id} & 0 \\
      0 \arrow{r}                        & \ZZ_5 \arrow{r}               & F_5  \arrow{r}            & \ZZ_5 \arrow{r}                 & 0
    \end{tikzcd}\]
  We now turn to KO-theory. A portion of the $E_2$-page of the AHSS is displayed in \cref{fig:B13:KO-AHSS}.
  \begin{figure}
    \begin{center}
    \begin{tikzpicture}[x=5ex,y=4ex]% as in grid below
      \renewcommand{\d}{|[fill=gray!15!white]|}      % diagonal
      \newcommand{\nr}[1]{{\text{\small\texttt{#1}}}}% number
      \newcommand{\X}{|[circle,thick,draw,red,inner sep=0pt, minimum size=1ex,text=gray]|} % delete
      \renewcommand{\S}{{\color{red} 2}} % pass to subgroup
      \matrix (m)
      [
      matrix of math nodes,
      nodes in empty cells,
      nodes={
        minimum width=5ex,
        minimum height=4ex,
        outer sep=-2pt},
      column sep=0ex,
      row sep=0ex
      ]
      {
        \quad\strut & \nr{0}  & \nr{1} & \nr{2}    & \nr{3} & \nr{4}    & \nr{5} & \nr{6} & \nr{7} & \nr{8}    & \nr{9}    & \nr{10} & \nr{11}   & \nr{12} & \nr{13}   \\
        \nr{0}     & \d{\ZZ} & 0      & \S{\ZZ}   & 0      & \ZZ       & 0      & \ZZ_5  & 0      & \ZZ_5     & \ZZ       & 0       & \S{\ZZ}  & 0       & \ZZ       \\
        \nr{-1}     & \ZZ_2   & \d{0}  & \X{\ZZ_2} & 0      & \X{\ZZ_2} & 0      & 0      & 0      & 0         & \ZZ_2     & 0       & \X{\ZZ_2} & 0       & \X{\ZZ_2} \\
        \nr{-2}     & \ZZ_2   & 0      & \d{\ZZ_2} & 0      & \X{\ZZ_2} & 0      & 0      & 0      & 0         & \ZZ_2     & 0       & \ZZ_2     & 0       & \X{\ZZ_2} \\
        \nr{-3}     & 0       & 0      & 0         & \d{0}  & 0         & 0      & 0      & 0      & 0         & 0         & 0       & 0  & 0       & 0         \\
        \nr{-4}     & \ZZ     & 0      & \ZZ       & 0      & \d{\ZZ}   & 0      & \ZZ_5  & 0      & \ZZ_5     & \ZZ       & 0       & \ZZ       & 0       & \ZZ       \\
        \nr{-5}     & 0       & 0      & 0         & 0      & 0         & \d{0}  & 0      & 0      & 0         & 0         & 0       & 0         & 0       & 0         \\
        \nr{-6}     & 0       & 0      & 0         & 0      & 0         & 0      & \d{0}  & 0      & 0         & 0         & 0       & 0  & 0       & 0         \\
        \nr{-7}     & 0       & 0      & 0         & 0      & 0         & 0      & 0      & \d{0}  & 0         & 0         & 0       & 0  & 0       & 0         \\
        \nr{-8}     & \ZZ     & 0      & \S{\ZZ}   & 0      & \ZZ       & 0      & \ZZ_5  & 0      & \d{\ZZ_5} & \ZZ       & 0       & \S{\ZZ}  & 0       & \ZZ       \\
        \nr{-9}     & \ZZ_2   & 0      & \X{\ZZ_2} & 0      & \X{\ZZ_2} & 0      & 0      & 0      & 0         & \d{\ZZ_2} & 0       & \X{\ZZ_2} & 0       & \X{\ZZ_2} \\
        \nr{-10}    & \ZZ_2   & 0      & \ZZ_2     & 0      & \X{\ZZ_2} & 0      & 0      & 0      & 0         & \ZZ_2     & \d{0}   & \ZZ_2     & 0       & \X{\ZZ_2} \\
        \nr{-11}    & 0       & 0      & 0         & 0      & 0         & 0      & 0      & 0      & 0         & 0         & 0       & \d{0}     & 0       & 0         \\
        \nr{-12}    & \ZZ     & 0      & \ZZ       & 0      & \ZZ       & 0      & \ZZ_5  & 0      & \ZZ_5     & \ZZ       & 0       & \ZZ       & \d{0}   & \ZZ       \\
        \nr{-13}    & 0       & 0      & 0         & 0      & 0         & 0      & 0      & 0      & 0         & 0         & 0       & 0         & 0       & \d{0}     \\
      };
      \path[arrow]
      % _____Differentials_____
      % d_2:
      (m-2-4)  edge[red,shorten >=1ex]  (m-3-6)
      (m-3-4)  edge[red,shorten >=1ex,shorten <=1ex]  (m-4-6)
      (m-10-4)  edge[red,shorten >=1ex]  (m-11-6)
      (m-11-4)  edge[red,shorten >=1ex,shorten <=1ex]  (m-12-6)
      (m-2-11) edge[lightgray,shorten >=1ex] (m-3-13)
      (m-3-11) edge[lightgray] (m-4-13)
      (m-10-11) edge[lightgray,shorten >=1ex] (m-11-13)
      (m-11-11) edge[lightgray] (m-12-13.west)
      (m-2-13) edge[red,shorten >=1ex]  (m-3-15)
      (m-3-13) edge[red,shorten >=1ex,shorten <=1ex]  (m-4-15)
      (m-10-13) edge[red,shorten >=1ex]  (m-11-15)
      (m-11-13) edge[red,shorten >=1ex,shorten <=1ex]  (m-12-15)
      % d_7:
      (m-6-4)  edge[lightgray] ($(m-12-11.west)!0.5!(m-12-11.north west)$)
      (m-6-6)  edge[lightgray] ($(m-12-13.west)!0.5!(m-12-13.north west)$)
      % d_9:
      (m-4-4.east)  edge[lightgray] (m-12-13)
      ;
      % _____Axes______________
      \path[arrow]
      (m-15-1.south east) edge ($(m-15-1.south east)+(0,15.5)$)%(m-1-1.north east)
      (m-1-1.south west) edge ($(m-1-1.south west)+(15,0)$) ;%(m-1-15.south east);
      \node at ($(m-1-1.south west)+(15.3,0)$) {\small\texttt{p}};
      \node at ($(m-15-1.south east)+(-0.4,15.4)$) {\small\texttt{q}};
    \end{tikzpicture}
    \end{center}
    \caption{The AHSS computing \(KO^*(B^{13})\), with shaded zero diagonal.
      The red arrows are the nonzero differentials,  and entries circled in red are those that vanish when passing to the \(E_3\)-page. The red twos indicate that \(\ZZ\) is replaced by the subgroup \(2\ZZ\) when passing to the \(E_3\)-page.}\label{fig:B13:KO-AHSS}
  \end{figure}
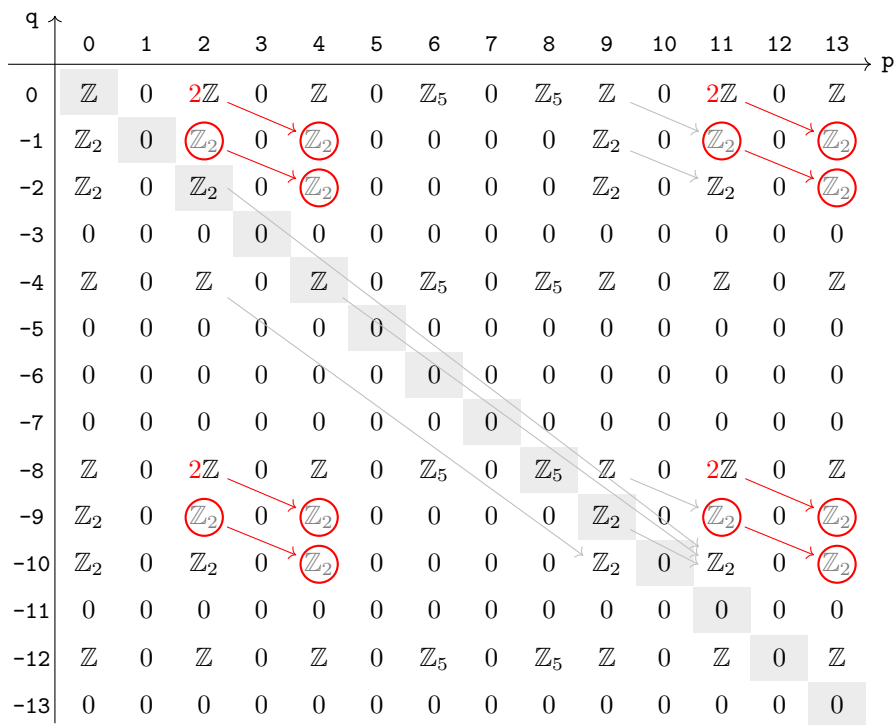
  The differentials \(d_2\) on the \(E_2\)-page are determined by \(Sq^2\), which we computed in \cref{B13:Sq2}.  \Cref{fig:B13:KO-AHSS} displays those \(d_2\)'s which are zero in gray and those which are not in red.  Once we have passed to the \(E_3\)-page, there are only three further differentials which could be nonzero, two \(d_7\)'s and one \(d_9\).  However, a comparison with the spectral sequence for \(K\)-theory shows that they all vanish.  Thus, the non-circled entries in \cref{fig:B13:KO-AHSS} constitute the \(E_\infty\)-page.

  The groups \(\reduced{KO}^i(B^{13})\) with \(i=-2,\dots,-6\) can be read off the spectral sequence directly. To compute \(KO^{-1}(B^{13})\), consider the morphism of spectral sequences induced by realification, and concentrate first on the diagonals computing \(K^0(B^{13})\) and \(KO^0(B^{13})\).  Compare the filtrations and the exact sequences that we obtain from the spectral sequences:
  \[\begin{tikzcd}
      0 \arrow{r}                         & F^{n+1} K \arrow{r}\arrow{d}{r} & F^{n}K \arrow{r}\arrow{d}{r}    & E_\infty^{n,-n}(K) \arrow{r} \arrow{d}{\rx_\infty} & 0 \\
      0 \arrow{r}                         & F^{n+1} KO \arrow{r}         & F^{n}KO \arrow{r}            & E_\infty^{n,-n}(KO) \arrow{r}          & 0
    \end{tikzcd}\]
  The vertical arrow on the right comes from the change-of-coefficients morphism \(H^n(-,K^{-n})\to H^n(-,KO^{-n})\) induced by the realification \(\rx \colon K^{-n}(pt) \to KO^{-n}(pt)\), described in detail in \cref{sec:AHSS}.  So we can compute cokernel and kernel of the realification map step by step. This gives us the following exact sequences:
  \begin{equation}
    \label{eq:B13:filtered-realification}
    \begin{aligned}
      0\to  0                                 & \to F^8K \to F^8KO \to \ZZ_2 \to 0                                                                       \\
      0\to  \ZZ_5                             & \to F^6K \to F^6KO \to \ZZ_2 \to 0                                                                       \\
      0\to  \ZZ_5                             & \to F^4K \to F^4KO \to \ZZ_2 \to 0                                                                       \\
      0\to  \ZZ\oplus\ZZ_5                    & \to F^2K \to F^2KO \to (\ZZ_2 \text{ or } 0) \to 0
    \end{aligned}
  \end{equation}
  Note that \(F^2 K = \reduced K\) and \(F^2 KO = \reduced{KO}\), so that we find:
  \begin{quote}
    The realification map \(\reduced{K}(B^{13}) \to \reduced{KO}(B^{13})\) has cokernel at most \(\ZZ_2\).
  \end{quote}
  Next, we analyse the map \(\eta\colon \reduced{KO}^{-1}(B^{13}) \to \reduced{KO}^0(B^{13})\), playing the same game with the filtrations as above. There is only one step:
  \[\begin{tikzcd}
      0 \arrow{r}                         & 0 \arrow{r}\arrow{d}         & F^{9}KO^0 \arrow{r}\arrow{d} & \ZZ_2 \arrow{r} \arrow{d}{\cong}       & 0 \\
      0 \arrow{r}                         & \ZZ \arrow{r}                & F^{9}KO^{-1} \arrow{r}       & \ZZ_2 \arrow{r}                        & 0
    \end{tikzcd}\]
  Thus, we find that
  \(
  \reduced{KO}^{-1}(B^{13}) \cong \ZZ \oplus \ZZ_2
  \).
  On the other hand, we know from above that \(\reduced{K}^{-1}(B^{13})\) is free.  So the Bott sequence
  \(
  \reduced{K} \to \reduced{KO} \to \reduced{KO}^{-1} \to \reduced{K}^{-1}
  \)
  implies:
  \begin{quote}
    The realification map \(\reduced{K}(B^{13}) \to \reduced{KO}(B^{13})\) has cokernel exactly \(\ZZ_2\).
  \end{quote}

  Next, to compute \(\reduced{KO}^{-7}(B^{13})\), we consider the realification \(\reduced{K}^{-6}(B^{13})\to \reduced{KO}^{-6}(B^{13})\). Note that the entries \(H^2(B^{13},\ZZ)\cong\ZZ\) in the rows \(q=0\) of the spectral sequence for KO-theory get replaced by the subgroup \(2\ZZ\) when passing to the \(E_3\)-page, so that \(\rx_\infty^{2,0}\) is surjective.  Again arguing filtration-step-by-filtration-step, we find that \(\reduced{K}^{-6}(B^{13})\to \reduced{KO}^{-6}(B^{13})\) is surjective.  The Bott sequence therefore implies that \(\reduced{KO}^{-7}(B^{13})\) injects into \(\reduced{K}^{-1}(B^{13})\), so \(\reduced{KO}^{-7}(B^{13})\) must be torsion-free.  Thus, we find that \(\reduced{KO}^{-7}(B^{13})\) is as displayed.

  Finally, for \(KO^0(B^{13})\), the spectral sequence initially only tells us that
  \(\reduced{KO}^0(B^{13}) \cong \ZZ\oplus\ZZ_5\oplus X\) with \(X=\ZZ_2\) or \(X=\ZZ_2\oplus\ZZ_2\) or \(X=\ZZ_4\).
  But we also find that \(\rx \colon \reduced{K}^1(B^{13})\to \reduced{KO}^1(B^{13})\) has cokernel \(\ZZ_2\).  As \(\reduced{K}^0(B^{13})\) contains no two-torsion, the Bott sequence implies that
  \(\reduced{KO}^0(B^{13})\) can contain at most one element of order a power of two.  Thus, \(\reduced{KO}^0(B^{13})\) is as displayed.

  The remaining maps in the Bott sequence are easily computed.  As \(\eta\in KO^{-1}(*)\) is two-torsion, the fact that most of the maps \(\eta\) are zero is clear.  For the two nonzero maps, we again use the spectral sequence and argue filtration-step-by-filtration-step.  The additional partial information concerning the maps \(\rx \) and \(\cx \) displayed in \cref{fig:B13:Bott} follows from the exactness of the sequence.
\end{proof}

\subsection{K-theory: multiplicative structure}
\begin{prop}\label{B13:K-ring}\label{B13:KO-ring}
  The real and complex K-rings of \(B^{13}\) are as follows:\footnote{The set of relations displayed for the complex K-ring is not minimal.}
  \begin{align*}
    K(B^{13}) &\cong \ZZ[u,y]/(5u^3,5u^4,u^5, y^3, uy^2, u^3y, y^2 - u^4, \\[-1ex]
              & \phantom{\cong\ZZ[u,y]/(} y - u^2 + u^3 - u^4, uy - u^3 + u^4, u^2y-u^4)\\
              &\cong \ZZ \cdot 1 \oplus \ZZ u \oplus \ZZ y \oplus \ZZ_5 u^3 \oplus \ZZ_5 u^4\\
    KO(B^{13})   & = \ZZ[y',w]/(5y'^2, 2w, w^2, wy', y'^3)                                              \\
                 & = \ZZ \cdot 1 \oplus \ZZ y' \oplus \ZZ_5 y'^2 \oplus \ZZ_2 w
  \end{align*}
  The ring homomorphisms \(\tx\colon K(B^{13})\to K(B^{13})\) and \(\cx \colon KO(B^{13}) \to K(B^{13})\) and the group homomorphism \(\rx \colon K(B^{13})\to KO(B^{13})\) are determined by:
  \begin{align*}
    \tx(y) &= y &&&&& \cx (y')        & = y        &  &&&                & \rx (u)    & = y'            & \rx (u^3)  & = 3y'^2       \\
    \tx(u) &= y-u &&&&& \cx (w)         & = 0        &  &&&                & \rx (y)    & = 2y'           & \rx (u^4)  & = 2y'^2
  \end{align*}
\end{prop}

\begin{proof}
  Let \(G=SU(5)\) and \(H=Sp(2)\times_{\ZZ_2}S^1\), so that \(B^{13} = G/H\).
  We first compute the complex K-ring.  By \cref{Improvement of Pittie's},
  \(K(B^{13})\) is isomorphic to \(R(H)\) modulo the ideal generated by the restrictions of reduced generators of \(R(G)\).
  Let us write
  \begin{align*}
    a     & :=(\cx' u)^2 & b       & :=x^2 & c  & := x^{-2} & d & :=(\cx' u)x & e & :=\lambda^2(\cx' u)
  \end{align*}
  for the generators of \(R(H)\) determined in \cref{B13:R(H)}.  Then in \(R(H)\) we have the relations displayed in line \eqref{rel1} below.
  Using \cref{B13:restricting-RG-to-RH}, we find that the restriction of reduced generators from \(R(G)\) to \(R(H)\) yield lines \eqref{rel2} and \eqref{rel3} as additional relations in \(K(B^{13})\):
  \begin{align}
   \label{rel1}
    bc    & = 1       & ab      & =d^2
 \\
    \label{rel2}
    d+c^2 & =5        & eb+dc^2 & = 10
 \\
    \label{rel3}
    ec+db & =10       & dc+b^2  & =5
  \end{align}
Thus, \(K(B^{13})\) is isomorphic to a quotient of the polynomial ring \(\ZZ [a, b, c,d,e ]\) by all relations above. In particular, we see that a possible set of additive generators is given by \(\{b^2, b, 1, c, c^2\}\).  Equivalently, a set of additive generators is given by \(\{b'^2, b', 1, c' , c'^2\}\), where now \(b' := b-1\) and \(c' := c-1\) both lie \(\reduced K(B^{13})\).  However, for our purposes, the following set of additive generators seems most suitable:
  \begin{align*}
    1   &                               &
    u   & :=  b'                        &
    u^3 & \;= 2b'^2  + c'^2 - 3b' - 3c' \\
    &&
    y   & :=  b' + c'                   &
    u^4 & \;= b'^2  + c'^2  - 2b' - 2c'
  \end{align*}
A tedious but straightforward calculation shows that these satisfy the relations and have the properties indicated above.

For our computation of the \(KO\)-ring, we will also need to understand in terms of our chosen generators the filtration \(F^i K := F^iK(B^{13})\) that appears in the AHSS.  We claim that this filtration is as follows:
\begin{equation}\label{eq:B13:K-filtration}
  \begin{aligned}
    F^1 K = F^2 K   & =  (u,       y )                                                                                    \\
    F^3 K = F^4 K   & = (u^2,     y )                                                                                     \\
    F^5 K = F^6 K   & =  (u^3, uy, y^2)                                                                                   \\
    F^7 K = F^8 K   & =  (u^4        )                                                                                    \\
    F^9 K &= 0
  \end{aligned}
\end{equation}
For \(F^1 K\), the kernel of the rank homomorphism, the claim is clear.  As the filtration is multiplicative, we moreover know ``\(\supset\)'' in each of the indicated equalities:
  \begin{compactitem}[]
  \item \(F^3 K = F^4 K\) must contain \((F^2 K)^2\), in particular \(u^2\) and \(uy = u^2 - y\), so \(u^2\) and \(y\).
  \item \(F^5 K = F^6 K\) must contain \(F^1 K\cdot F^2 K\), in particular \(u^3\), \(uy\) and \(y^2\).
  \item \(F^7 K = F^8 K\) must contain \((F^1 K)^4\), in particular \(u^4\).
  \end{compactitem}
  In our specific situation, equality in \eqref{eq:B13:K-filtration} will follow if we can show that, assuming equality, each quotient \(F^i K/F^{i+1} K\) is isomorphic to the corresponding entry of the \(E_2=E_{\infty}\)-page of the AHSS.  This is indeed the case:
  \begin{align*}
    F^2 K / F^4 K &\cong \ZZ \cdot u \\
    F^4 K / F^6 K &\cong \ZZ \cdot \{u^2, y\}/F^6 K \cong \ZZ\cdot u^2  & \text{ (since \(u^2 = y\) modulo \(uy)\) }\\
    F^6 K / F^8 K &\cong \ZZ \cdot \{u^3, uy, u^2y, u^3y\} / F^8 K \cong \ZZ_5 \cdot  u^3 \\
    F^8 K &\cong \ZZ_5 \cdot u^4
  \end{align*}
  For \(F^6 K /F^8 K\) note that \(uy = u^3 \mod F^8 K\):  We know that \(u^2 = uy + y\), so multiplying by \(u\) we obtain \(u^3 = u^2y + uy\).  Using that \(u^2y = y^2\), we can rewrite this as:
  \(
  u^3 = y^2 + uy
  \).
  The other possible generators of \(F^6 K / F^8 K\) vanish: \(u^2y = y^2\) vanishes modulo \(F^8 K\), and \(u^3y = y^3 = 0\) anyway.

  We now turn to the KO-ring.  The associated graded ring that the AHSS converges to has no nontrivial products, so this ring is not very helpful. We therefore run through the additive computation again, choosing names for all generators and taking a note of where they map under complexification.

  First, we have \(F^9 KO = \ZZ_2 w\) for some mysterious element \(w\) mapping to zero under complexification.  Next we know that \(F^8 KO \cong \ZZ_5 \oplus \ZZ_2\), and that \(F^8 K \to F^8 KO\) has cokernel \(\ZZ_2\).  So the five-torsion in \(F^8 KO\) must be generated by some element of the form \(\rx (y^2)\).  The image of \(\rx (y^2)\) under \(\cx \) is \(2y^2\), since \(y^2\) is self-conjugate.  Equivalently:
  \[
    F^8 KO = \ZZ_5 y'' \oplus \ZZ_2 w
  \]
  for \(y'' := 3r(y^2)\), and \(\cx (y'') = y^2\).
  The next filtration step for \(\cx \colon KO \to K\) is:
  \[\xymatrix@R=6pt{
      \ZZ_2 w           \oplus \ZZ_5 y''  \ar[r] \ar[d] & F^4 KO                     \ar[r]\ar[d]    & \ZZ \ar[d]^{2} \\
      \ZZ_5 u^3 \oplus \ZZ_5 y^2  \ar[r]       & \ZZ y \oplus \ZZ_5 u^3 \oplus \ZZ_5 y^2 \ar[r] & \ZZ y
    }\]
  The vertical map on the right is \(\cx \colon KSp(*) \to K(*)\), which is multiplication by \(2\).
  So \(F^4 KO\)  contains an element mapping to \(2y + \alpha u^3 + \beta y^2\) for some \(\alpha\), \(\beta\).
  As \(y''\) maps to \(y^2\), there must also be an element mapping to \(2y + \alpha u^3\).
  As the image of any element under complexification is self-conjugate, we must have \(\alpha=0\).
  We conclude that there must be some element \(v \in F^4 KO\) such that \(\cx (v) = 2y\).
  So \(F^4 KO = \ZZ v \oplus \ZZ_5 y'' \oplus \ZZ_2 w\).

  Finally, \(F^2 KO = \reduced{KO}\), \(F^2 K = \reduced K\), and we have seen that the cokernel of \(\cx \colon \reduced{KO} \to \reduced{K}\) has no two-torsion.  This implies that there is an element \(y'\in F^2 KO\) such that \(\cx (y') = y\).  Consider the two short exact sequences corresponding to this last filtration step:
  \[\xymatrix@R=9pt{
      {\ZZ u \phantom{\;\;\oplus \ZZ u^3} \oplus \ZZ_5 y'' \oplus \ZZ_2 w}  \ar[r]\ar[d] & {F^2 KO}                                              \ar[r]\ar[d] & {\ZZ_2 y'} \ar[d]^{0}\\
      {\ZZ  y \oplus \ZZ u^3 \oplus \ZZ_5 y^2 \phantom{\;\oplus \ZZ_2 w}} \ar[r]         & {\ZZ u \oplus \ZZ y \oplus \ZZ_5 u^3 \oplus \ZZ_5 y^2} \ar[r]      & {\ZZ u}
    }\]
  The short exact sequence along the top gives us:
  \[
    F^2 KO = \ZZ y' \oplus \ZZ_5 y'' \oplus \ZZ_2 w.
  \]
  As \(\cx (2y') = \cx (v)\), we have \(v = 2y' + \alpha w\) for some \(\alpha\).  As \(w\) is in the smallest piece of the filtration, we can replace the generator \(v\) by \(2y'\) everywhere above. Thus, altogether, the filtration on \(KO(B^{13})\) may be written as follows:
  {\newcommand{\soplus}{\;\oplus\;}
    \begin{alignat*}{13}
      F^1 KO = F^2 KO  & = (y',w)      &  & = \ZZ\cdot y'  & \soplus & \ZZ_5\cdot y'^2 & \soplus & \ZZ_2 \cdot w \\
      F^3 KO = F^4 KO  & = (2y',w)     &  & = \ZZ\cdot 2y' & \soplus & \ZZ_5\cdot y'^2 & \soplus & \ZZ_2 \cdot w \\
      F^5 KO = F^8 KO  & = (y'^2,w)    &  & =              &         & \ZZ_5\cdot y'^2 & \soplus & \ZZ_2 \cdot w \\
      F^9 KO           & = (w)         &  & =              &         &                 &         & \ZZ_2 \cdot w \\
      F^{10} KO        & = 0
    \end{alignat*}
  }%
  It remains to determine the products of \(y'\), \(y''\) and \(w\).  For filtration reasons, we have \((y')^3 = (y'')^2 = w^2 = y'w = y''w  = 0\).  As \(\cx ((y')^2) = \cx (y'')\), we moreover obtain the relation \((y')^2 = y'' + \delta w\) for some \(\delta \in \{0,1\}\).  We will return to the value of \(\delta\) at the end of the proof.

  The behaviour of \(\cx\) has already been determined.  It follows that \(\rx(y) = \rx(\cx y') = 2y'\) and likewise \(\rx(y^2) = \rx (\cx y'') = 2y''\).  From \(\cx\rx(u) = u + \tx u = y = \cx(y')\), we deduce
  \(
    \rx(u) = y' + \alpha w
  \)
  for some \(\alpha \in \ZZ_2\).  As we already know that \(w\) does not lie in the image of \(\rx\) (\cf \eqref{eq:B13:filtered-realification} in the additive computation), \(\alpha\) must be zero.  Similarly, \(\cx\rx(u^3) = u^3 + \tx(u^3) = u^3 + (y-u)^3 = \dots = 3y^2\) implies \(\rx(u^3) = 3y''\).  Now  \((y')^2 = \rx(u)y' = \rx(u\cdot \cx y')\) and  \(y''=\rx(3y^2)\) are both in the image of \(\rx\), so their difference \(\delta w\) must also be in the image of \(\rx\).  But again, we already know that \(w\) is not in the image of \(\rx\). So \(\delta = 0\), \ie \((y')^2 = y''\).
\end{proof}

\subsection{The multiplicative map \texorpdfstring{$\phi$}{phi}}
For any Lie group \(G\), there is a multiplicative map \[\phi\colon R(G) \to RO(G)\] such that \(\cx \phi(x) = x (\tx x)\) for any element \(x\) \cite{Adams:LieLectures}*{Lemma 7.2}.  Likewise, we have such map \(\phi\colon K(X) \to KO(X)\).  The idea is that for any complex vector bundle \(E\), there is a canonical conjugate linear automorphism \(J\) on \(E\otimes \tx{E}\) such that \(J^2=\id\).  More precisely, we have the following:
\begin{prop}[Bousfield]\label{properties-of-phi}
  For any finite CW complex \(X\), there is a natural map
  \[
    \phi\colon K(X) \to KO(X)
  \]
  such that for all \(x, y \in K(X)\):
  \begin{align*}
    \phi(xy)    & = \phi(x)\phi(y)                                                \\
    \phi(x+y)   & = \phi(x) + \phi(y) + \rx ((\tx x) y)                           \\
    \cx \phi(x) & = x(\tx x)                                                      \\
    \phi(\tx x) & = \phi(x)                       \\
    \phi(x)     & = \lambda^2(\rx x) - \rx(\lambda^2 x)
  \end{align*}
\end{prop}
\begin{proof}[Reference]
  See \cite{Bousfield:2-primary}:  the existence of \(\phi\) and the first three properties are asserted at the beginning of 6.11 and in Thm~6.5.  The expression for \(\phi(x)\) for \(x\in K^0(X)\) in terms of exterior powers is implicit in Thm~6.7, see 6.2 (iv). The relation \(\phi(\tx x) = \phi(x)\) follows.
\end{proof}

Note that the above properties allow us to easily compute \(\phi\) \emph{modulo the kernel of \(\cx \)}.  For example, in the case of \(X=B^{13}\), we must have \(\phi(y) = y'^2 + \delta w \) for some \(\delta\in\{0,1\}\).  In fact, \(\delta = 0\) in this case:

\begin{lem}\label{B13:phi-in-r}
  The map  \(\phi\colon K(B^{13}) \to KO(B^{13})\) is determined by
  \begin{align*}
    \phi(u) &= -y'\\
    \phi(y) &= y'^2
  \end{align*}
  and the above properties. In particular, the image of \(\phi\) is contained in the image of \(\rx\), and hence does \emph{not} contain the two-torsion summand of \(KO(B^{13})\).
\end{lem}

\begin{proof}
  \(\phi(u)\) can be computed as follows.
  It is easy to compute \(\phi\) on representation rings because there \(\cx \) is injective, so \(\phi\) is completely determined by \(\cx \phi(x) = x(\tx x)\).  For the generator \(x\in R(S^1)\), we find \(\phi(x) = 1\), so \(\phi(x^2)=1\) for the element \(x^2 \in R(H)\).
  Moreover, \(\phi\) is compatible with the maps \(\alpha\), \ie \(\phi\circ\alpha=\alpha_O\circ\phi\).  Thus, we find:
  \begin{align*}
    \phi(u) &= \phi(\alpha(x^2-1))
              = \alpha_O(\phi(x^2-1))
              = \alpha_O(2-\rx (x^2)) \\
            &= \alpha_O(\rx(1-x^2))
              = \rx (\alpha(1-x^2))
              = \rx (-u) \\
            &= -y'
  \end{align*}
  The value of \(\phi(y)\) can be deduced from this:
  \begin{align*}
    \phi(y) &= \phi(u+\tx u) = \phi(u) + \phi(\tx u) + \rx (u^2)\\
            &= 2\phi(u) + \rx (y+u^3-u^4) = -2y' + 2y' + y'^2\\
            &= y'^2
  \end{align*}
  Finally, note that \(K(B^{13})\) is generated by \(u\) and \(y\) as a ring.  Both \(\phi(u)\) and \(\phi(y)\) lie in the image of \(\rx \).  The properties of \(\phi\) therefore imply that the whole image of \(\phi\) is contained in the image of \(\rx \).  The image of \(\rx \) does not contain the two-torsion.
\end{proof}

\subsection{Non-surjectivity of \texorpdfstring{$\alpha_O$}{alpha-O}}
\begin{prop}\label{B13:image-of-alphaO}
  In the notation of \cref{B13:KO-ring}, the image of the ring homomorphism \(\alpha_O\colon RO(Sp(2)\times_{\ZZ_2}S^1)\to KO(B^{13})\) is the subring
  \[
    S := \ZZ\cdot 1 \oplus \ZZ y'\oplus \ZZ_5 y'^2
  \]
  generated by \(y'\).  As a subgroup, \(S\) is of index two.
\end{prop}
Note that the product of \(y'\) with the additional generator \(w\) of \(KO(B^{13})\) is trivial.  Thus, if we take \(\reduced{S}:=\ZZ y'\oplus \ZZ_5 y'^2\) to denote the rank zero ideal of \(S\), we obtain the splitting
\(\reduced{KO}(B^{13}) \cong \reduced{S} \times \ZZ_2 w\) alluded to in \cref{THM:B13}.

\begin{proof}
By \cref{B13:KO-ring}, the subring \(S\) is additively generated by the multiplicative unit \(1\) and by the image \(\im(\rx)\) of the realification map \(\rx \colon K(B^{13})\to KO(B^{13})\). For the inclusion \(S\subset\im(\alpha_O)\), it therefore suffices to note that \(1\in\im(\alpha_O)\) (as \(\alpha_O\) is a ring homomorphism) and \(\im(\rx)\subset \im(\alpha_O)\) (as the complex version \(\alpha\) is surjective by \cref{Improvement of Pittie's} and \(\alpha_O \rx = \rx \alpha\) by \cref{alpha:c-r-t-q}).  It remains to show the converse inclusion: the image of \(\alpha_O\) is contained in the subgroup of \(KO(B^{13})\) generated by \(1\) and \(\im(\rx )\).

  By \cref{B13:RO(H)-generators}, \(RO(H)\) is generated as a ring by \(1\), \(u^2\), \(\lambda^2 u\) and the image of the realification map. Clearly, \(\alpha_O\) maps the image of the realification map to \(\im(\rx )\), and it maps \(1\) to \(1\).  So it suffices to show that \(\alpha_O\) also sends the two additional generators \(u^2\) and \(\lambda^2 u\) to \(\im(\rx )\).  As noted in \cref{alpha:exterior-powers}, the different flavours of \(\alpha\) are compatible with multiplication and exterior powers.  In particular,
  \(\alpha_O(u^2) = \alpha_{Sp}(u)^2\) and \(\alpha_O(\lambda^2u) = \lambda^2(\alpha_{Sp}u)\).
  It therefore suffices to show that:
  \begin{alignat*}{8}
    KSp(B^{13})\cdot KSp(B^{13}) &\subset \im(\rx ) &\quad& \text{ in } KO(B^{13})\\
    \lambda^2(KSp(B^{13})) &\subset \im(\rx ) && \text{ in } KO(B^{13})
  \end{alignat*}
  As the quaternionification \(\qx \colon K(B^{13})\to KSp(B^{13})\) is surjective (\cref{B13:K-additive}), we only need to show that \(\im \qx  \cdot \im \qx  \subset \im \rx \) and \(\lambda^2(\im \qx )\subset \im \rx \).  The first inclusion is clear from the usual multiplication rules for \(\rx \): explicitly, \(\qx x\cdot \qx y = \rx (x\cdot \cx' \qx y)\).
  For the second inclusion, note that, for any \(X\) and any \(z\in K(X)\), we have
  \(\lambda^2(\qx z) = \lambda^2(\rx z) = \rx (\lambda^2 z) + \phi(z)\) (\cref{properties-of-phi}).
  For \(X=B^{13}\), \cref{B13:phi-in-r} shows that \(\im\phi \subset \im \rx \), so the claim follows.
\end{proof}

\subsection{Pontryagin classes}
\begin{prop}\label{B13:Pontryagin}
  In the notation of \cref{B13:KO-ring}, the generator \(w\in \reduced{KO}(B^{13})\) is the unique nonzero element with trivial Pontryagin classes.
\end{prop}
\begin{proof}
  We first compute the total Chern class \(c=1+c_1+c_2+\dots\) of the generators \(u\) and \(y\) of \(K(B^{13})\).  We know that \(b=u+1\) represents a line bundle, and the spectral sequence computations imply \(c_1(b) = c_1(u) = \pm \beta\).  So \(c(u) = 1 \pm \beta\), and we deduce:
  \begin{align*}
    c(y)   &= c(u + \tx u) = 1-\beta^2\\
    c(y^2) &= c(b^2 + \tx b^2 -4b- 4\tx b + 6)
           = (1\pm 2\beta)(1\mp 2\beta)/\left((1\pm \beta)^4(1\mp \beta)^4\right)\\
           &= (1 -4\beta^2)/(1-\beta^2)^4
           = 1 - 6\beta^4\
           = 1 - \beta^4
  \end{align*}

  The only nontrivial Pontryagin classes that the cohomology of \(B^{13}\) admits are \(p_1\), with values in \(\ZZ \beta^2\),  and \(p_2\), with values in \(\ZZ_5\beta^4\).  For the total Pontryagin classes \(p=1+p_1+p_2\) we find:
  \begin{alignat*}{10}
    p(y') &= 1-c_2(y)+c_4(y) &&= 1 + \beta^2 &&\\
    p(y'^2) &= 1-c_2(y^2) + c_4(y^2) &&= 1 - \beta^4 &&\\
    p(w) &= 1-c_2(\cx (w))+c_4(\cx (w)) &&= 1  &&\text{ (as \(\cx(w) = 0\)) }
  \end{alignat*}
  So for an arbitrary element \(\mu y' + \nu y'^2 + \delta w \in \reduced{KO}(B^{13})\), we find:
  \begin{align*}
    p(\mu y' + \nu y'^2 +\delta w)
    &= (1+\beta^2)^\mu(1-\beta^4)^\nu \\
    &= (1+\mu\beta^2 + \textstyle\binom{\mu}{2}\beta^4)(1-\nu\beta^4)\\
    &= 1 + \mu\beta^2 +(\textstyle\binom{\mu}{2}-\nu)\beta^4 \in 1 + \ZZ\beta^2 + \ZZ_5 \beta^4
  \end{align*}
  This is equal to \(1\) if and only if \(\mu=0\) in \(\ZZ\) and \(\nu=0\) in \(\ZZ_5\).
\end{proof}

%%% Local Variables:
%%% mode: latex
%%% TeX-master: "main"
%%% End: